\documentclass[12pt,a4paper]{amsart}
\usepackage[nocompress]{cite}
\usepackage{listings}
\usepackage{docmute}
\usepackage{amssymb, enumerate, calligra}
\usepackage{amsthm}
\usepackage{amsmath}
\usepackage{amsxtra}
\usepackage{latexsym}
\usepackage{mathrsfs}
\usepackage[all,cmtip]{xy}
\usepackage[all]{xy}
\usepackage{enumitem}
\usepackage{xcolor}
\usepackage{graphicx}
\usepackage{comment}
\usepackage{mathabx,epsfig}

\def\racts{\ \rotatebox[origin=c]{90}{$\circlearrowleft$}\ }

\usepackage[dvipdfmx]{hyperref}
\usepackage{cleveref}

\usepackage{mathtools}

\DeclarePairedDelimiter\floor{\lfloor}{\rfloor}

\newtheorem{thm}{Theorem}[section]
\crefname{thm}{Theorem}{Theorems}
\newtheorem{lem}[thm]{Lemma}
\crefname{lem}{Lemma}{Lemmas}
\newtheorem{conj}[thm]{Conjecture}
\crefname{conj}{Conjecture}{Conjectures}
\newtheorem{claim}[thm]{Claim}
\crefname{claim}{Claim}{Claims}
\newtheorem{prop}[thm]{Proposition}
\crefname{prop}{Proposition}{Propositions}
\newtheorem{cor}[thm]{Corollary}
\crefname{cor}{Corollary}{Corollaries}

\crefname{property}{Property}{Properties}
\newtheorem{que}[thm]{Question}
\crefname{que}{Question}{Questions}

\theoremstyle{definition}
\newtheorem{defn}[thm]{Definition}
\crefname{defn}{Definition}{Definitions}
\newtheorem{rmk}[thm]{Remark}
\crefname{rmk}{Remark}{Remarks}
\newtheorem*{ack}{Acknowledgements}
\numberwithin{equation}{section}

\newtheorem{ex}[thm]{Example}
\crefname{ex}{Example}{Examples}

\def\C{{\mathbb C}}
\def\Q{{\mathbb Q}}
\def\R{{\mathbb R}}
\def\Z{{\mathbb Z}}
\def\P{{\mathbb P}}
\def\A{{\mathbb A}}

\def\QQ{\overline{\mathbb Q}}
\def\p{{ \mathfrak{p}}}     %%added
\def\O{{ \mathcal{O}}}
\def\m{{ \mathfrak{m}}} 
\def\a{{ \mathfrak{a}}}  
\def\n{{ \mathfrak{n}}}  
\def\I{{ \mathcal{I}}}
\def\J{{ \mathcal{J}}}

\def\F{{ \mathcal{F}}}
\def\KK{ \overline{K}}

\DeclareMathOperator{\pr}{pr}

\DeclareMathOperator{\id}{id}

\DeclareMathOperator{\Spec}{Spec}
\DeclareMathOperator{\Exc}{Exc}
\DeclareMathOperator{\Supp}{Supp}

\DeclareMathOperator{\mult}{mult} 
\DeclareMathOperator{\lct}{lct} 
\DeclareMathOperator{\codim}{codim}

\newenvironment{claimproof}[0]
  {%
   \paragraph{\it Proof.}%
  }
  {%
    \hfill$\blacksquare$%
  }

  {\end{list}}

\title[Growth of local heights]
{Growth of local height functions along orbits of self-morphisms on projective varieties}
\author{Yohsuke Matsuzawa}
\address{Department of Mathematics, Rikkyo University, 3-34-1 Nishi-Ikebukuro, Toshima-ku, Tokyo, 171-8501 JAPAN}

\email{\href{mailto:matsuzaway@rikkyo.ac.jp}{matsuzaway@rikkyo.ac.jp}}

\begin{document}

\begin{abstract}
We consider the limit 
\[
\lim_{n\to \infty} \sum_{v\in S} \lambda_{Y,v}(f^{n}(x))/h_{H}(f^{n}(x))
\]
where $f \colon X \longrightarrow X$ is a surjective self-morphism on a smooth projective variety $X$
over a number field, $S$ is a finite set of places,
$ \lambda_{Y,v}$ is a local height function associated with a proper closed subscheme $Y \subset X$,
and $h_{H}$ is an ample height function on $X$.
We give a geometric condition which ensures that the limit is zero,
unconditionally when $\dim Y=0$ and assuming Vojta's conjecture when $\dim Y\geq1$. 
In particular, we prove (one is unconditional, one is assuming Vojta's conjecture) Dynamical Lang-Siegel type theorems, 
that is, the relative sizes of coordinates of orbits on $\P^{N}$ are asymptotically the same with 
trivial exceptions.
These results are higher dimensional generalization of Silverman's classical result.
\end{abstract}

\maketitle

\setcounter{tocdepth}{1}
\tableofcontents

\section{Introduction}

In this paper, an {\it algebraic scheme} over a field $k$ means a separated scheme of finite type over $k$.
A {\it variety} over $k$ is an algebraic scheme over $k$ which is irreducible and reduced.
A {\it nice variety} over $k$ is a smooth projective geometrically irreducible scheme over $k$.

Let $K$ be a number field.
The following question has a fundamental importance in the study of arithmetic dynamics 
of self-morphisms of (higher dimensional) algebraic varieties.

\begin{que}\label{mainQ}
Let $X$ be a nice variety over $K$ and $f \colon X \longrightarrow X$ be a surjective morphism.
Let $Y \subset X$ be a proper closed subscheme of $X$.
Let us fix a local height function $\{ \lambda_{Y,v} \}_{v \in M_{K}}$ associated with $Y$,
a global height function $h_{Y}$ associated with $Y$, and 
a height function $h_{H}$ associated with an ample divisor $H$ on $X$.
Let $x \in X(K)$ be a point. 
\begin{enumerate}
\item\label{q:glhtvsamp}
Under what conditions on $f, Y$, and $x$ do we have
\begin{align*}
\lim_{n \to \infty}\frac{h_{Y}(f^{n}(x))}{h_{H}(f^{n}(x))} =0 \quad \text{\rm ?}
\end{align*}

\item\label{q:lochtvsamp}
Let $S \subset M_{K}$ be any finite set of places.
Under what conditions on $f, Y$, and $x$ do we have
\begin{align*}
\lim_{n \to \infty}\frac{\sum_{v\in S} \lambda_{Y,v}(f^{n}(x))}{h_{H}(f^{n}(x))} =0 \quad \text{\rm ?}
\end{align*}

\end{enumerate}
\end{que}

Let us give some comments on \cref{mainQ}.
Since $h_{Y}$ and $ \lambda_{Y,v}$ are getting smaller as $\dim Y$ is getting smaller,
the limits in the question tend to be $0$ for smaller dimensional $Y$.
When $\codim Y\geq 2$, it would be reasonable to expect that $h_{Y}(f^{n}(x))$
does not get lager as the ample height $h_{H}(f^{n}(x))$ grows (with trivial exceptions).
Any positive answer to \cref{mainQ}(\ref{q:glhtvsamp}) would be helpful to understand,
for example, the growth of functions which are pull-backs of ample height functions by rational maps
($Y$ would be the indeterminacy locus of the rational map). 
Also, as pointed out by Silverman, when, 
\begin{align*}
&X=\P^{2}_{\Q} \quad \text{and} \\
& f \colon \P^{2} \longrightarrow \P^{2}; (x:y:z) \mapsto (ax:by:z)
\end{align*}
where $a,b$ are multiplicatively independent integers,
and $Y=\{(1:1:1)\}$, then
\[
\lim_{n\to \infty} \frac{h_{Y}(f^{n}(1:1:1))}{h_{H}(f^{n}(1:1:1))}=0
\]
is equivalent to 
\[
\lim_{n \to \infty}\frac{\log \gcd(a^{n}-1,b^{n}-1)}{n}=0
\]
which is a highly non-trivial theorem in Diophantine approximation by Bugeaud, Corvaja, and Zannier \cite{bcz}.

On the other hand, when $\codim Y=1$, answers to  \cref{mainQ}(\ref{q:glhtvsamp}) could be
purely algebro geometric.
If $Y$ can be (birationally) contracted to a higher codimensional subvariety, 
the problem reduces to a higher codimensional one (for possibly rational self-maps).
If $Y$ has enough positivity (e.g.\ ample), the limit in \cref{mainQ}(\ref{q:glhtvsamp}) would never be zero.

While \cref{mainQ}(\ref{q:glhtvsamp}) seems not to be studied extensively so far,
there are some works related to \cref{mainQ}(\ref{q:lochtvsamp}) (when $Y$ is a divisor).
When $Y$ is an ample divisor, a positive answer to \cref{mainQ}(\ref{q:lochtvsamp})
implies there are only finitely many $S$-integral points with respect to $Y$ in the $f$-orbit of $x$.
For example, when $X=\P^{1}$, Silverman proved that the limit in \cref{mainQ}(\ref{q:lochtvsamp}) 
is zero with trivial exceptions \cite{sil93}.
For $X=\P^{N}$, there are works by Yasufuku \cite{yas1,yas2} under the assumption of Vojta's conjecture.

In this paper, we focus on \cref{mainQ} (\ref{q:lochtvsamp}). 
We give sufficient conditions for the limit to be zero in terms of geometry,
completely unconditional when $\dim Y=0$ and assuming  Vojta's conjecture when $\dim Y>0$.
Let us introduce some notion which we need to state our main theorems.

\subsection{Arithmetic degrees}

\begin{defn}
Let $K$ be a number field and $X$ be a nice variety over $K$.
Let $h_{H}$ be a height function on $X$ associated with an ample divisor $H$ on $X$.
Let $f \colon X \longrightarrow X$ be a surjective morphism.
For any point $x \in X(\KK)$, \emph{the arithmetic degree of $f$ at $x$} is 
\[
\alpha_{f}(x):= \lim_{n\to \infty} \max\{1, h_{H}(f^{n}(x))\}^{1/n}.
\]
This limit always exists and is independent of the choice of $H$ and $h_{H}$
(cf. \cite{ks3, ks1}) .
\end{defn}

\begin{rmk}
Let $d_{1}(f)$ be the first dynamical degree of $f$, i.e.\ 
the maximum modulus of eigenvalues of $f^{*} \colon N^{1}(X_{\KK}) \longrightarrow N^{1}(X_{\KK})$,
where $N^{1}(X_{\KK})$ is the group of divisors modulo numerical equivalence.
Then it is known that $ \alpha_{f}(x) \leq d_{1}(f)$ for all $x \in X(\KK)$ (\cite{ks1,ma}) and
conjectured that the equality holds if $x$ has Zariski dense $f$-orbit (Kawaguchi-Silverman conjecture).
\end{rmk}

\begin{ex}
When $X= \P^{N}_{K}$, $ \alpha_{f}(x)= d_{1}(f)$ if the $f$-orbit of $x$ is infinite and $ \alpha_{f}(x)=1$ otherwise.
Here the first dynamical degree $d_{1}(f)$ is just the degree of the coprime homogeneous polynomials
defining $f$ in this case.
\end{ex}

Arithmetic degree measures the asymptotic growth rate of $h_{H}(f^{n}(x))$.
Indeed, we have the following due to Sano.

\begin{prop}[{\cite[Theorem 1.1]{sano18}}]\label{prop:growthht}
Suppose $ \alpha_{f}(x) >1$.
Then there are a non-negative integer $l$ and positive real numbers $C_{1}, C_{2}$ such that
\[
C_{1}n^{l} \alpha_{f}(x)^{n} \leq \max\{1, h_{H}(f^{n}(x))\} \leq C_{2} n^{l} \alpha_{f}(x)^{n} 
\]
for all $n\geq 1$. 
\end{prop}
%$C_{1}, C_{2}$ depend on height of $x$.

\subsection{Multiplicities}
 
Let $k$ be a field of characteristic zero.

\begin{defn}
Let $f \colon X \longrightarrow Y$ be a finite flat morphism between algebraic schemes.
For a (scheme) point $x \in X$, we define the  \emph{multiplicity of $f$ at $x$} as
\[
e_{f}(x) = l_{ \mathcal{O}_{X,x}}(\mathcal{O}_{X,x}/f^{*} \mathfrak{m}_{f(x)} \O_{X,x}).
\]
Here $\m_{f(x)}$ is the maximal ideal of $\O_{Y,f(x)}$ and 
 $l_{ \mathcal{O}_{X,x}}$ stands for the length as an $ \mathcal{O}_{X,x}$ module.
\end{defn}

The following quantity plays a key role in this paper.

\begin{defn}
Let $f \colon X \longrightarrow X$ be a finite flat self-morphism of an algebraic scheme over $k$.
Let $x \in X$ be a scheme point.
We write
\[
e_{f, +}(x) := e_{+}(x) := \lim_{n\to \infty }e_{f^{n}}(x)^{1/n}
\]
where the existence of the limit is due to Favre (see \cite[Theorem 2.5.8]{fav}, \cite[\S 7]{gig14}, \cite{gig}).  
For a subvariety $P\subset X$ with generic point $\eta$,
we also write $e_{+}(P) = e_{+}(\eta)$.
See \cref{thm:gignac} (\ref{e+}) for more properties of this invariant.
\end{defn}

Note that if $x$ is $f$-periodic of period $r$, this is just the geometric mean of the multiplicities of $f$ on the orbit:
\[
e_{+}(x) = \left(e_{f}(x)e_{f}(f(x))e_{f}(f^{2}(x))\cdots e_{f}(f^{r-1}(x))\right)^{1/r}.
\]

\begin{ex}
If $f$ is unramified, $e_{f^{n}}(x)=1$ for all $x \in X$ and $n \geq 1$.
Hence $e_{f,+}(x) = 1$ for all $x \in X$.
\end{ex}

\begin{ex}
Let $f \colon \P^{1}_{k} \longrightarrow \P^{1}_{k}$ be a surjective morphism.
Let $x \in \P^{1}_{k}$ be a scheme point.
Let $O_{f}(x)$ be the $f$-orbit of $x$.
Then we have
\begin{align*}
e_{f,+}(x)
\begin{cases}
=1 \quad \txt{\small if either $x$ is the generic point, \\
\small $O_{f}(x)$ is infinite, or $O_{f}(x)$ is finite and the cycle\\
\small  does not contain any ramification points of $f$}\\[3mm]
>1 \quad \text{\small otherwise.}
\end{cases}
\end{align*}
\end{ex}

\begin{ex}\label{ex:P2pcf}
Let $f \colon \P^{2}_{k} \longrightarrow \P^{2}_{k}$ be the morphism defined by
\[
f(X:Y:Z) = ((X-2Y)^{2} : X^{2} : (X-2Z)^{2}).
\]
Then the ramification locus of $f$ is the union of three lines:
\[
(X=0) \cup (X=2Y) \cup (X=2Z).
\]
All these three lines are preperiodic under $f$ and the orbits are the following:
 \[
\xymatrix{
(X=2Y) \ar[r] &  (X=0) \ar[r] & (Y=0) \ar[r] & (X=Y) \racts \\
(X=2Z) \ar[r] & (Z=0) \ar[r] & (Y=Z) \ar@/^10pt/[r]<-1ex> & (X=Z) \ar@/^10pt/[l]<-1ex>.
}
\]
By this, we can easily see that if the orbit of a scheme point $x \in \P^{2}_{k}$ 
hits the ramification locus infinitely many times, then $x$ must be equal to the closed point $(0:0:1)$.
Hence we get
\begin{align*}
e_{f,+}(x)=
\begin{cases}
4 \quad \text{if $x=(0:0:1)$}\\
1 \quad \text{otherwise.}
\end{cases}
\end{align*}

\end{ex}

\subsection{Main theorems}
Now we state our main theorems.
Let $K$ be a number field.

The first theorem concerns the case when $\dim Y=0$.
We state the theorem using arithmetic distance functions $\delta_{X,v}(x,y)$ because we want to make it compatible with
a classical theorem in \cite{sil93} by Silverman.
Roughly speaking, arithmetic distance function is $\delta(x,y) \approx -\log (\text{$v$-adic distance between $x$ and $y$})$.
See \cref{sec:lochtdef} for the definition of arithmetic distance function.

\begin{thm}\label{thm:locdistvsht}
Let $X$ be a nice variety over $K$.
Let $f \colon X \longrightarrow X$ be a finite surjective morphism and $S \subset M_{K}$ a finite set.
Let $h_{H}$ be an ample height function on $X$ and $ \delta_{X}$ be an arithmetic distance function on $X$.

Let $x, y \in X(K)$ be points satisfying the following:
\begin{enumerate}
\item $ \alpha_{f}(x)>1$;
\item 
For every $f$-periodic subvariety $P \subset X$ such that $y \in P$, we have $e_{f,+}(P) < \alpha_{f}(x)$.
\end{enumerate}
Then 
\[
\lim_{n\to \infty}\frac{\sum_{v\in S} \delta_{X,v}(f^{n}(x),y)}{h_{H}(f^{n}(x))}=0.
\]
\end{thm}

\begin{rmk}\label{rmk:basicrmk}\

\begin{enumerate}

\item The function $\delta_{X,v}(-,y)$ where $y$ is a fixed $K$-point is just
a local height function $ \lambda_{y,v}$ associated with the closed subset $\{y\}$.

\item If the assumption $ \alpha_{f}(x)>1$ holds, then the orbit $O_{f}(x)$ is infinite and  the first dynamical degree of $f$ is larger than one: 
$d_{1}(f)>1$.
\item If $x$ has Zariski dense $f$-orbit, then the Kawaguchi-Silverman conjecture implies that $ \alpha_{f}(x) = d_{1}(f)$.
\item 
The self-morphism $f$ is called  \emph{int-amplified} if there is an ample divisor $A$ on $X_{\KK}$ such that
$(f_{\KK})^{*}A - A$ is ample. (the subscript $(\ )_{\KK}$ stands for base change.)
This is equivalent to saying that all eigenvalues of $(f_{\KK})^{*} \colon N^{1}(X_{\KK}) \longrightarrow N^{1}(X_{\KK})$
have modulus strictly greater than $1$.
If $f$ is int-amplified, then $ \alpha_{f}(x)>1$ if and only if the $f$-orbit $O_{f}(x)$ of $x$ is infinite (cf. \cite[Theorem 1.1]{sano17}). 

\item \label{rmk:onassumption}
If an $f$-periodic subvariety $P$ with generic point $\eta$ satisfies $e_{f,+}(P)>1$, 
then at least one of $f^{i}(\eta)$ is contained in the ramification divisor $R_{f}$ of $f$, i.e.\ $f^{i}(P) \subset R_{f}$. 
Thus to check the assumption (2), we only have to look at the cycles of periodic subvarieties contained in $R_{f}$.

The assumption (2) is equivalent to $e_{f,-}(y) < \alpha_{f}(x)$.
See \cref{thm:gignac} (\ref{e-}) for the definition of $e_{f,-}$ and  \cref{rmk:twocond} for the equivalence.
Since $e_{f,-}$ is upper semicontinuos (by \cref{thm:gignac}(4), this is due to Gignac) and $e_{f,-}=1$ at the generic point of $X$, 
there is an open dense subset $U \subset X$ such that 
\[
e_{f,-}(y) < \alpha_{f}(x) \quad \text{for $y \in U(K)$}.
\]
In particular, general point $y$ satisfies the assumption.

\item As pointed out in \cite{sil93} (comments below Theorem D),
\cref{thm:locdistvsht} is far from being true in transcendental setting.
That is, for example, if $v$ is an archimedean absolute value such that $K_{v}=\C$,
$x \in X(\C)$, and $h_{H}(f^{n}(x))$ is replaced with $d_{1}(f)^{n}$,
we can also ask if the sequence $ \delta_{X,v}(f^{n}(x), y)/d_{1}(f)^{n}$ goes to zero.
This fails for uncountably many $x$ even for $y$ with $e_{f,-}(y)=1$.
\end{enumerate}

\end{rmk}

For special cases, the assumptions in \cref{thm:locdistvsht} become simple.

\begin{cor}
In the setting of  \cref{thm:locdistvsht}, assume $f$ is {\'e}tale.
Let $x\in X(K)$ be a point such that $ \alpha_{f}(x)>1$.
Then for all $y \in X(K)$, we have
\[
\lim_{n\to \infty}\frac{\sum_{v\in S} \delta_{X,v}(f^{n}(x),y)}{h_{H}(f^{n}(x))}=0.
\]
\end{cor}
\begin{proof}
Since $f$ is \'etale, we automatically have $e_{f,+}(P)=1$ for all $f$-periodic subvariety $P \subset X$.
Hence the condition (2) in \cref{thm:locdistvsht}  vacuously holds.
\end{proof}

We also recover Silverman's theorem (\cite[Theorem E]{sil93}). 

\begin{cor}
Let $f \colon \P^{1} \longrightarrow \P^{1}$ a surjective morphism over $K$ with $\deg f \geq 2$.
Let $S \subset M_{K}$ be a finite set.
Let $h$ be the naive height on $\P^{1}$ and $ \delta$ be the arithmetic distance function on $\P^{1}$.
Let $x, y \in \P^{1}(K)$ be points and suppose:
\begin{enumerate}
\item the $f$-orbit $O_{f}(x)$ of $x$ is infinite;
\item $y$ is not a totally ramified periodic point of $f$.
\end{enumerate}
Then we have
\[
\lim_{n\to \infty}\frac{\sum_{v\in S} \delta_{v}(f^{n}(x),y)}{h(f^{n}(x))}=0.
\]
\end{cor}
\begin{proof}
Let $d = \deg f$.
We check the two conditions in \cref{thm:locdistvsht}.
Since $O_{f}(x)$ is infinite, we have $ \alpha_{f}(x)=d >1$ and (1) is satisfied.
The condition (2) in \cref{thm:locdistvsht} is equivalent to saying that $y$ is not a totally ramified periodic point of $f$.
\end{proof}

Next, we move to the case where $\dim Y\geq 0$.
\begin{defn}
A subset $O \subset X$ is called  \emph{generic} if $O\cap Z$ is finite for all proper Zariski closed subsets $Z \subset X$.
\end{defn}

Assuming Vojta's conjecture, we can prove the following.

\begin{thm}\label{thm:divlcht}
Let $X$ be a nice variety over $K$.
Let $f \colon X \longrightarrow X$ be a finite surjective morphism and $S \subset M_{K}$ a finite set.
Let $h_{H}$ be an ample height function on $X$.
Let $Y \subset X$ be a proper closed subscheme and $x \in X(K)$ a point.
Suppose:
\begin{enumerate}
\item
$ \alpha_{f}(x)>1$;
\item
the $f$-orbit $O_{f}(x)$ of $x$ is generic;
\item
For every $f$-periodic subvariety $P$ such that $Y \cap P \neq  \emptyset$, we have $e_{f,+}(P)< \alpha_{f}(x)$.
\end{enumerate}

Assume Vojta's conjecture {\rm (}for blow ups of $X${\rm )}.
Then 
\[
\lim_{n\to \infty} \frac{\sum_{v\in S} \lambda_{Y,v}(f^{n}(x))}{h_{H}(f^{n}(x))}=0.
\]

\end{thm}

\begin{rmk}\

\begin{enumerate}
\item
If $P \subset X$ is an $f$-periodic subvariety with $e_{f,+}(P)>1$,
then $f^{i}(P) \subset R_{f}$ for some $i\geq 0$, where $R_{f}$ is the ramification divisor of $f$.
\item
If the Dynamical Mordell-Lang conjecture is true for $f$,
then $O_{f}(x)$ is generic if $O_{f}(x)$ is Zariski dense.
\item
The assumption (3) is equivalent to 
\[
\max\{e_{f,-}(y) \mid \text{$y\in Y$ closed point}\} < \alpha_{f}(x).
\]
See \cref{sec:mult} for the definition of $e_{f,-}$.
Since $e_{f,-}$ is upper semicontinuous (by \cref{thm:gignac}(4)) and is equal to $1$ at the generic point, 
there is an open dense subset $U \subset X$ such that $Y$ satisfies the assumption (3) if and only if
$Y \subset U$.
\end{enumerate}
\end{rmk}

Without the genericness of $O_{f}(x)$, we can still prove the following.

\begin{thm}\label{thm:divlchtnongen}
Let $X$ be a nice variety over $K$.
Let $f \colon X \longrightarrow X$ be a finite surjective morphism and $S \subset M_{K}$ a finite set.
Let $h_{H}$ be an ample height function on $X$.
Let $Y \subset X$ be a proper closed subscheme and $x \in X(K)$ a point.
Suppose:
\begin{enumerate}
\item
$ \alpha_{f}(x)>1$;
\item
For every $f$-periodic subvariety $P$ such that $Y \cap P \neq  \emptyset$, we have $e_{f,+}(P) < \alpha_{f}(x)$.
\end{enumerate}

Let
\[
d= \liminf_{n\to \infty} \frac{h_{Y}(f^{n}(x))}{h_{H}(f^{n}(x))}.
\]
Assume Vojta's conjecture (for blow ups of $X$).
Then for any $ \epsilon>0$, the set
\begin{align*}
\left\{ z \in O_{f}(x) \middle| \frac{\sum_{v\notin S} \lambda_{Y,v}(z)}{h_{H}(z)} \leq d - \epsilon \right\}
\end{align*}
is not Zariski dense in $X$.
\end{thm}

When $X=\P^{N}$ and $Y$ is a divisor $D$, this can be stated in the following way.

\begin{thm}\label{thm:PNcase}
Let $f \colon \P^{N}_{K} \longrightarrow \P^{N}_{K}$ be a finite surjective morphism with first dynamical degree $d_{1}(f)\geq 2$.
Let $S \subset M_{K}$ be a finite set.
Let $D$ be an effective Cartier divisor on $\P^{N}$ and $x \in \P^{N}(K)$ a point.
Suppose for every $f$-periodic subvariety $P \subset \P^{N}$ such that $D \cap P \neq  \emptyset$, we have $e_{f,+}(P) < d_{1}(f)$.

Assume Vojta's conjecture {\rm (}for blow ups of $\P^{N}${\rm )}.
Then for any $ \epsilon>0$, the set
\begin{align*}
\left\{ z \in O_{f}(x) \middle| \frac{\sum_{v\notin S} \lambda_{D,v}(z)}{h_{\P^{N}}(z)} \leq \deg D - \epsilon \right\}
\end{align*}
is not Zariski dense.
Here $h_{\P^{N}}$ is the naive height function on $\P^{N}$ and $\deg D$ is the degree of $D$ with respect to $\O_{\P^{N}}(1)$.
 
\end{thm}
\begin{proof}
We may assume $O_{f}(x)$ is infinite.
Then $ \alpha_{f}(x)= d_{1}(f)$.
Thus $f, x, D$ satisfy the assumptions in \cref{thm:divlchtnongen}.
Let $H$ be a hyperplane in $\P^{N}$.
Since $D \sim (\deg D) H$, 
\[
\lim_{n\to \infty} h_{D}(f^{n}(x))/h_{H}(f^{n}(x)) = \deg D
\]
and we are done by \cref{thm:divlchtnongen}.
\end{proof}

\begin{rmk}
Yasufuku proved very similar results to \cref{thm:PNcase}.
For example see \cite[Theorem 3]{yas1}, \cite[Theorem 7]{yas2}.
Yasufuku assumes Vojta's conjecture for certain simple normal crossing (SNC for short) divisors on $\P^{N}$, while we assume Vojta's conjecture for
blow ups of $\P^{N}$ in \cref{thm:PNcase} (for certain SNC divisors).
Also, we assume a condition on $D$ concerning its geometric position with respect to the periodic ramified loci of $f$
and as a consequence, we get a simpler bound ``$\deg D - \epsilon$" .

If we ignore the difference of the use of Vojta's conjecture, \cref{thm:PNcase} contains Yasufuku's theorems when the divisor $D$
satisfies the assumption of \cref{thm:PNcase}.
When $D$ dose not satisfy the assumption of \cref{thm:PNcase},
it seems that the pull-backs of $D$ by the self-morphism $f$ tend not to contain normal crossing divisors 
and therefore Yasufuku's theorems become weaker (their conclusion could be empty).
We are not sure, however, how to make this precise. It seems it is not easy to deduce some information about normal crossing part
of pull-backs of $D$ from the asymptotic invariants.
\end{rmk}

\begin{rmk}
\cref{thm:PNcase} implies in particular that any sets of $(D,S)$-integral points in the orbit $O_{f}(x)$ are not Zariski dense.
Here a set of $(D,S)$-integral points means a subset of $(\P^{N}\setminus D)(K)$ on which $\sum_{v\notin S} \lambda_{D,v}$
is bounded.
See \cite{ly19} for a result when $N=2$ (assuming the Lang-Vojta conjecture). %他にも文献あれば追加．
\end{rmk}

\begin{rmk}
According to the terminology in \cite{hssil}, the conclusion of \cref{thm:PNcase} can be rephrased that
the set of quasi-integral points in the orbit $O_{f}(x)$ is not Zariski dense.
\end{rmk}

As an immediate corollary of \cref{thm:divlcht,thm:PNcase}, we get the following.

\begin{cor}[Dynamical Lang-Siegel for $\P^{N}_{\Q}$]\label{cor:dls}
Let $\P^{N}_{\Q}$ be the projective space over $\Q$ with coordinate $x_{0}, \dots, x_{N}$ with $N\geq 1$.
Let $f \colon \P^{N}_{\Q} \longrightarrow \P^{N}_{\Q}$ be a surjective morphism with first dynamical degree $d_{1}(f)\geq 2$.
Suppose for every $f$-periodic subvariety $P$ such that $(x_{0}=0) \cap P \neq  \emptyset$, we have $e_{f,+}(P) < d_{1}(f)$.

Assume Vojta's conjecture {\rm (}for blow ups of $\P^{N}${\rm )}.
Let $x \in \P^{N}(\Q)$ be a point and write
\[
f^{n}(x) = (a_{0}(n):\cdots : a_{N}(n))
\]
with $a_{0}(n),\dots, a_{N}(n) \in \Z$ and $\gcd(a_{0}(n),\dots, a_{N}(n))=1 $.
If $x$ has infinite $f$-orbit, then for any $ \epsilon>0$, the set 
\[
\left\{  f^{n}(x) \middle| n\geq 0, \frac{\log|a_{0}(n)|}{\log \max\{|a_{0}(n)|,\dots, |a_{N}(n)|\}} < 1- \epsilon\right\}
\]
is not Zariski dense in $\P^{N}$.
If $O_{f}(x)$ is generic, we have
\[
\lim_{n\to \infty} \frac{\log|a_{0}(n)|}{\log \max\{|a_{0}(n)|,\dots, |a_{N}(n)|\}} =1.
\]

\end{cor}

\begin{proof}
Let $D$ be the divisor on $\P^{N}$ defined by $x_{0}=0$.
Then we can take a local height function  associated with $D$ as
\[
\lambda_{D, p}(a_{0}: \cdots : a_{N}) = \log \left( \frac{\max\{|a_{0}|_{p}, \dots, |a_{N}|_{p}\}}{|a_{0}|_{p}}\right) \qquad \text{$p=\infty$ or prime}
\]
for $(a_{0}:\cdots : a_{N})\in \P^{N}(\Q)$ with $a_{0}\neq 0$.
Let us write simply $|\ |_{\infty} = |\ |$.

Then if $a_{0}(n) \neq 0$, we have
\begin{align*}
&\frac{\sum_{p\neq \infty} \lambda_{D,p}(f^{n}(x))}{h_{\P^{N}}(f^{n}(x))}\\[3mm]
&= \frac{\log \max\{|a_{0}(n)|,\dots, |a_{N}(n)|\} - \log \left(\frac{\max\{|a_{0}(n)|, \dots, |a_{N}(n)|\}}{|a_{0}(n)|} \right)}{\log \max\{|a_{0}(n)|,\dots, |a_{N}(n)|\}}\\[3mm]
&= \frac{\log|a_{0}(n)|}{\log \max\{|a_{0}(n)|,\dots, |a_{N}(n)|\}}.
\end{align*}
Thus the first statement follows from \cref{thm:PNcase}.
For the second statement, suppose $O_{f}(x)$ is generic.
Then by \cref{thm:divlcht}, we have
\begin{align*}
0 = & \lim_{n\to \infty}\frac{ \lambda_{D, \infty}(f^{n}(x))}{h_{\P^{N}}(f^{n}(x))} =
\lim_{n\to \infty} \frac{\log \left( \frac{\max\{|a_{0}(n)|, \dots, |a_{N}(n)|\}}{|a_{0}(n)|}\right) }{ \log \max\{|a_{0}(n)|,\dots, |a_{N}(n)|\} }\\
& = \lim_{n\to \infty} \left( 1- \frac{\log|a_{0}(n)|}{\log \max\{|a_{0}(n)|,\dots, |a_{N}(n)|\}} \right)
\end{align*}
and we are done.
\end{proof}

\begin{rmk}
\cref{cor:dls} says that Vojta's conjecture implies the so called Dynamical Lang-Siegel conjecture for $\P^{N}_{\Q}$ ( \cite[Conjecture 21.4]{adsurv}).
Note that the assumption in \cite[Conjecture 21.4]{adsurv} is not sufficient and there is a counter example to that form,
cf. \cref{ex:cetodls}.
\end{rmk}

\noindent
{\bf Organization of the paper}
In \cref{sec:lochtdef,sec:roth}, we fix notation related to absolute values and local height functions, 
and recall a version of Roth's theorem.
In \cref{sec:mult}, we review a theorem on asymptotic behavior of multiplicities of self-morphisms, which is due to Favre and Gignac.
In \cref{sec:ado}, we prove \cref{thm:locdistvsht}.
In \cref{sec:sco}, we prove an (unconditional) theorem on the size of coordinates of orbits of self-morphisms on projective space.
We also give a counter example to \cite[Conjecture 21.4]{adsurv}.
In \cref{sec:lochtdiv}, we prove \cref{thm:divlcht,thm:divlchtnongen}.
We give a lower bound for the sequence of log canonical thresholds $\lct(X, (f^{n})^{-1}(Y))$, which enable us to apply Vojta's conjecture
effectively.
Only in this section do we use Vojta's conjecture.

We give several (non-)examples in \cref{sec:ado,sec:sco,sec:lochtdiv} of our main theorems.

\noindent
{\bf Convention}
In this paper, we work over a number field or a field of characteristic zero.
\begin{itemize}
\item An  \emph{algebraic scheme} over a field $k$ is a separated scheme of finite type over $k$;
\item A  \emph{variety} over $k$ is an algebraic scheme over $k$ which is irreducible and reduced;
\item A  \emph{nice variety} over a field $k$ is a smooth projective geometrically irreducible scheme over $k$;
\item Let $X$ be a scheme over a field $k$ and $k \subset k'$ be a field extension. 
 \emph{The base change} $X \times_{\Spec k}\Spec k'$ is denoted by $X_{k'}$.
For an ``object" $A$ on $X$, we sometimes use the notation $A_{k'}$ to express the base change of $A$ to $k'$ 
without mentioning to the definition of the base change if the meaning is clear.
\item For a closed subscheme $Y \subset X$, the  \emph{ideal sheaf defining $Y$} is denoted by $\I_{Y}$;
\item For a self-morphism $f \colon X \longrightarrow X$ of an algebraic scheme over $k$ and a 
point $x$ of $X$ (scheme point or $k'$-valued point where $k'$ is a field containing $k$), the  \emph{$f$-orbit of $x$}
is denoted by $O_{f}(x)$, i.e. $O_{f}(x) = \{ f^{n}(x) \mid n=0,1,2, \dots\}$.
\end{itemize}

\begin{ack}
The author would like to thank Joseph Silverman for discussing this subject with him and 
giving him many suggestions and valuable comments.
He would also like to thank Kenta Hashizume, Reimi Irokawa, Takumi Murayama, Kenta Sato,
Yuya Takeuchi, Takehiko Yasuda,  and Shou Yoshikawa for answering his questions.
The author would like to thank the referee for many suggestive comments.
The author is supported by JSPS Overseas Research Fellowship.
He would also like to thank the department of mathematics at Brown University for hosting him during his 
fellowship.
\end{ack}

\section{Local heights and arithmetic distance function}\label{sec:lochtdef}

We fix notation related to local height functions and arithmetic distance functions.
See \cite{bg, Lan,hs} for the definitions and basic properties of absolute values and local/global height functions associated with divisors,
and see \cite{sil87} for local/global height associated with subschemes and arithmetic distance functions.

Let $K$ be a field with proper set of absolute values $M_{K}$. 
Let $M( \overline{K})$ be the set of absolute values on $ \overline{K}$ which extend absolute values of $M_{K}$.
For any intermediate field $K \subset L \subset \KK$ with $[L:K]<\infty$, 
let $M(L)$ be the set of absolute values on $ L$ which extend absolute values of $M_{K}$.
Note that $M(L)$ is also a proper set of absolute values.

Let $X$ be a projective variety over $K$ and $Y\subset X_{\KK}$ be a proper closed subscheme.
We can equip $Y$ with a function, which is called  \emph{the local height function associated with $Y$}:
\begin{align*}
\lambda_{Y} \colon (X\setminus Y)(\KK) \times M(\KK) \longrightarrow \R; (x,v) \mapsto \lambda_{Y,v}(x).
\end{align*}
Note that this is determined up to $M_{K}$-bounded function.

For a particular choice of the local height function, each $ \lambda_{Y,v}$ is a function on $(X\setminus Y)( \overline{K})$.
If we write  $Y=D_{1}\cap \cdots \cap D_{r}$ for some effective Cartier divisors $D_{i}$, then
\[
 \lambda_{Y,v} = \min_{1\leq i \leq r}\{ \lambda_{D_{i}, v} \} \quad \text{up to $M_{K}$-bounded function}
\]
where $ \lambda_{D_{i}, v}$ are the usual  (logarithmic) local heights associated with $D_{i}$.

When $Y$ is also defined over $K$, let us denote by $Y$ the model over $K$.
In this case, we can choose $ \lambda_{Y}$ so that the indicated map exists,
\[
\xymatrix{
(X \setminus Y)(L) \times M(\KK) \ar@{^{(}->}[r] \ar[d]_{\id \times (\ )|_{L}} & 
(X\setminus Y)(\KK) \times M(\KK) \ar[r]^(.80){ \lambda_{Y}} & \R \\
(X\setminus Y)(L) \times M(L) \ar@/_{10pt}/[rru]_{\exists} &&
}
\]
for any intermediate field $K \subset L \subset \KK$ with $[L:K]<\infty$.
The induced map $(X\setminus Y)(L)\times M(L) \longrightarrow \R$ is also denoted by $ \lambda_{Y}$:
the image of $(x,v) \in (X\setminus Y)(L)\times M(L)$ is denoted by $ \lambda_{Y,v}(x)$.
A global height function $h_{Y} \colon (X \setminus Y)(\KK) \longrightarrow \R$
associated with $Y$ is defined by this choice of $ \lambda_{Y.v}$:
\begin{align}\label{eq:glhtY}
h_{Y}(x) = \frac{1}{[L:K]}\sum_{v \in M(L)} [L_{v}:K_{v|_{K}}] \lambda_{Y, v}(x)
\end{align}
for $x \in (X \setminus Y)(L)$.

The  \emph{arithmetic distance function on $X$} is the local height function on $X \times X$ associated with the diagonal $ \Delta$:
\[
 \delta_{X, v}  := \lambda_{ \Delta, v} \qquad \text{for $v \in M(\KK)$}.
\]

In this paper, 
{\bf when $K$ is a number field, $M_{K}$ is the set of absolute values that are normalized as in \cite[p11 (1.6)]{bg}. }
Namely, if $K=\Q$, then $M_{\Q}=\{|\ |_{p} \mid \text{$p=\infty$ or a prime number} \}$ with
\begin{align*}
 &|a|_{\infty} =
 \begin{cases}
 a \quad \text{if $a\geq0$}\\
 -a \quad \text{if $a<0$}
 \end{cases}
 \\
 & |a|_{p} = p^{-n} \quad \txt{if $p$ is a prime and  $a=p^{n}\frac{k}{l}$ where\\ $k,l$ are non zero integers coprime to $p$.}
\end{align*}
For a number field $K$, $M_{K}$ consists of the following absolute values:
\begin{align*}
|a|_{v} = |N_{K_{v}/\Q_{p}}(a)|_{p}^{1/[K:\Q]} 
\end{align*}
where $v$ is a place of $K$ which restricts to $p = \infty$ or a prime number.
We use this normalization to define global height functions $h_{H}$, $h_{Y}$, etc.
Note that under this notation, (\ref{eq:glhtY}) becomes
\begin{align*}
h_{Y}(x) = \sum_{v \in M_{L}} \lambda_{Y, v}(x)
\end{align*}
where $ \lambda_{Y,v}$ is defined by using normalized absolute values in $M_{L}$.

Let $K$ is a number field.
Let $v \in M_{K}$.
For any two extensions $v', v'' \in M( \overline{K})$ of $v$, 
$ \lambda_{Y, v'}- \lambda_{Y, v''}$ is a bounded function on $(X\setminus Y)(K)$.
Let $ \C_{v}$ be the completion of the algebraic closure of $K_{v}$.
Let $\|\ \|$ be the absolute value on $\C_{v}$ which extends $v$.
Then $ \{ \|\ \| \}$ is a proper set of absolute values and we can define $ \lambda_{Y_{\C_{v}}, \|\ \|}$.
Fix an embedding $ \overline{K} \subset \C_{v}$ over $K$ and let $v'$ be the restriction of $\|\ \|$ on $ \overline{K}$ via this embedding.
Then $ \lambda_{Y_{\C_{v}}, \|\ \|}- \lambda_{Y, v'}$ is a bounded function on $(X\setminus Y)( \overline{K})$,
where $ \lambda_{Y_{\C_{v}}, \|\ \|}$ is considered as a function on $(X\setminus Y)( \overline{K})$ via the embedding $ \overline{K} \subset \C_{v}$.
We will denote $\lambda_{Y_{\C_{v}}, \|\ \|}$ by $ \lambda_{Y, \|\ \|}$ for simplicity.

These hold also for $ \delta_{X, v}$ and we will use the same abbreviation.

\section{Diophantine approximation}\label{sec:roth}

%%%%%%%%%%%
\if0
\begin{thm}
Let $X$ be a projective variety over $K$ of dimension $n$.
Let $S\subset M_{K}$ be a finite set.
For each $v\in S$, let $Y_{0,v},\dots , Y_{n,v}$ be closed subschemes of $X_{ \overline{K}}$ in general position.
Let $\|\ \|_{v}$ be an absolute value on $ \overline{K}$ which extends $|\ |_{v}$ for $v\in S$.
Fix local height functions $ \lambda_{Y_{i,v}, \|\ \|_{v}}$.
Let $H$ be an ample divisor on $X$ and fix an ample height $h_{H}$.
Let $ \epsilon_{Y_{i,v}}(H)$ be the Seshadri constant.

Then for any $ \epsilon>0$, there is a proper Zariski closed subset $Z \subset X$ such that
\[
\sum_{v\in S} \sum_{i=0}^{n} \epsilon_{Y_{i,v}}(H) \lambda_{Y_{i,v}, \|\ \|_{v}}(x) \leq (n+1+ \epsilon)h_{H}(x)
\]
for all $x\in X(K)\setminus Z$.
\end{thm}
\fi
%%%%%%%%%%%

We use the following version of Roth's theorem as a key input to the proof of \cref{thm:locdistvsht}.

\begin{thm}\label{Roth}
Let $X$ be a projective variety over a number field $K$.
Let $H$ be a very ample divisor on $X$, and fix a height function $h_{H}$.
Let $y\in X( \overline{K})$ be any point.
Fix an arithmetic distance function $ \delta_{X}$.
Then for any $v\in M(\KK)$ and for any $ \epsilon>0$, there is a finite subset $Z(y, \epsilon) \subset X(K)$ such that
\begin{align*}
\delta_{X, v}(x,y) \leq (2+ \epsilon)h_{H}(x) \quad \text{for all $x \in X(K)\setminus Z(y, \epsilon)$}.
\end{align*}
\end{thm}

\begin{proof}
This follows from, for example, \cite[Corollary 6.4]{mr15}.
Note that in \cite{mr15}, they use different notation from this paper.
After adjusting the difference of normalizations of absolute values, we have
\begin{align*}
\delta_{X,v}(x,y) = -\log d_{v}(x,y) \quad \text{up to $M_{K}$-bounded function}
\end{align*}
where $d_{v}$ is the distance function used in \cite{mr15}.
Also, since we take $H$ to be very ample, which corresponds to $L$ in \cite[Corollary 6.4]{mr15},
the Seshadri constant in \cite[Corollary 6.4]{mr15} is $\geq 1$: $ \epsilon_{x}(L)\geq 1$.
\end{proof}

\section{Multiplicities}\label{sec:mult}
In this section, we review a theorem by Dinh, Favre, and Gignac on the multiplicities of preimages of points 
under self-morphisms.
This is the key to bound the singularities of $(f^{n})^{-1}(Y)$ 
in the proof of our main theorems.

In this section, the ground field is a field of characteristic zero.
In this section, if we write $x \in X$ for a scheme $X$, this literally means $x$ is a point of the underlying topological space of $X$.

\begin{defn}
For a finite flat morphism $f \colon X \longrightarrow Y$ between algebraic schemes and a (scheme) point $x \in X$,
we define  \emph{the multiplicity of $f$ at $x$} by
\[
e_{f}(x) = l_{ \mathcal{O}_{X,x}}(\mathcal{O}_{X,x}/f^{*} \mathfrak{m}_{f(x)}\O_{X,x}).
\]
Here $l_{ \mathcal{O}_{X,x}}$ stands for the length as an $ \mathcal{O}_{X,x}$ module.
\end{defn}

\begin{rmk}\label{rmk:multfldext}
Let $k$ be the ground field, which is characteristic zero as we always assume, and $k \subset k'$ be a field extension.
Then for any scheme point $x' \in X_{k'}$ lying over $x \in X$, we have
$e_{f_{k'}}(x')=e_{f}(x)$. 
The proof of this fact requires several steps and it is too long to include here.
When $k'$ is algebraic over $k$ (only this case we need this fact), however, it is easy to see this.
Suppose $k'$ is algebraic over $k$.
Take a finite intermediate extension $k \subset K \subset k'$.
Consider the following commutative diagram:
 \[
\xymatrix{
X_{k'} \ar[r]^{f_{k'}} \ar[d]_{ \alpha} & Y_{k'} \ar[d]^{\beta}\\
X_{K} \ar[r]^{f_{K} } \ar[d]_{p} & Y_{K} \ar[d]^{q}\\
X \ar[r]_{f} & Y.
}
\]
Set $x'' = \alpha(x')$, $y'=f_{k'}(x')$, and $y'' = \beta(y')$.
If $K$ is large enough, we have $\m_{x''} \O_{X',x'} = \m_{x'}$ and $\m_{y''} \O_{Y',y'} = \m_{y'}$.
Note also that $p$ and $q$ are finite \'etale morphisms since $k$ is characteristic zero.
Use \cref{claim:chainrule} for each square and we can show $e_{f_{k'}}(x')=e_{f}(x)$.

For a point $x \in X(k')$, let $ \xi \in X$ be the image of $x \colon \Spec k' \to X$
and $ \xi'$ be the image of $(x, \id) \colon \Spec k' \to X_{k'}$.
Then we have $e_{f}( \xi)= e_{f}( \xi')$ and define $e_{f}(x)= e_{f}( \xi)= e_{f}( \xi')$.
\end{rmk}

\begin{lem}\label{lem:chainrule}
Let $f \colon X \longrightarrow Y$ and $g \colon Y \longrightarrow Z$ be finite flat morphisms between algebraic schemes.
Let $x \in X$.
Then we have
\[
e_{g\circ f}(x)= e_{f}(x)e_{g}(f(x)).
\]
\end{lem}
\begin{proof}
This follows from the following claim:
\begin{claim}\label{claim:chainrule}
Let $(A,\m_{A}) \longrightarrow (B,\m_{B}) \longrightarrow (C, \m_{C})$ be local homomorphisms between Noetherian local rings such that
$\m_{A}B$ is $\m_{B}$-primary and $\m_{B}C$ is $\m_{C}$ primary.
Suppose $B \longrightarrow C$ is flat.
Then
\[ 
l_{C}(C/\m_{A}C) = l_{C}(C/\m_{B}C) l_{B}(B/\m_{A}B).
\]
\end{claim}
\begin{claimproof}
Since $\m_{A}B$ is $\m_{B}$-primary and $\m_{B}C$ is $\m_{C}$-primary, 
$l_{B}(B/\m_{A}B)$ and $l_{C}(C/\m_{B}C)$ are finite.
By the exact sequence of $C$-modules
\[
0 \longrightarrow \m_{B}C/\m_{A}C \longrightarrow C/\m_{A}C \longrightarrow C/\m_{B}C \longrightarrow0 
\]
we have 
\[
l_{C}(C/\m_{A}C) = l_{C}(\m_{B}C/\m_{A}C) + l_{C}(C/\m_{B}C).
\]
Since $B \longrightarrow C $ is flat we have the following exact sequence of $C$-modules:
\begin{align}\label{eq:exactseq}
0 \longrightarrow \m_{A}B {\otimes}_{B}C \longrightarrow \m_{B} {\otimes}_{B}C \longrightarrow (\m_{B}/\m_{A}B) {\otimes}_{B}C \longrightarrow 0.
\end{align}
Thus we have
\begin{align*}
&l_{C}(\m_{B}C/\m_{A}C) = l_{C}(\m_{B} {\otimes}_{B}C/\m_{A}B {\otimes}_{B}C) & \text{since $B \to C$ is flat}\\
&= l_{C}((\m_{B}/\m_{A}B) {\otimes}_{B}C) & \text{by  (\ref{eq:exactseq})}\\
&= l_{B}(\m_{B}/\m_{A}B)l_{C}(C/\m_{B}C)    & \text{since $B \to C$ is flat}.
\end{align*}
Thus we get 
\[
l_{C}(C/\m_{A}C) = l_{C}(C/\m_{B}C) (l_{B}(\m_{B}/\m_{A}B)+1) =  l_{C}(C/\m_{B}C) l_{B}(B/\m_{A}B).
\]

\end{claimproof}

\end{proof}

\begin{lem}\label{lem:multvslen}
Let $f \colon X \longrightarrow Y$ be a finite flat morphism between algebraic schemes.
Let $x \in X$ and $y=f(x)$.
Then 
\[
\m_{y} \O_{X,x} \supset \m_{x}^{e_{f}(x)}.
\]
\end{lem}
\begin{proof}
Let $A=\O_{X,x}/\m_{y} \O_{X,x}$.
Since $f$ is finite, this is an Artin local ring. Let $\m \subset A$ be the maximal ideal.
Note that $e_{f}(x) = l_{\O_{X,x}}(A)=l_{A}(A)$ since $\O_{X,x}$-submodules of $A$ are exactly $A$-submodules of $A$.
Let $r>0$ be the minimum integer such that $\m^{r}=0$. 
Then $\m^{i}/\m^{i+1} \neq 0$ for $i=0,\dots, r-1$ by Nakayama's lemma.
Therefore, we have $r\leq l_{A}(A) = e_{f}(x)$ and get $\m^{e_{f}(x)}=0$.
This means $\m_{y} \O_{X,x} \supset \m_{x}^{e_{f}(x)}$.
\end{proof}

A function $\varphi \colon X \longrightarrow \R$ is called  \emph{upper semicontinuous} if
$\{x \in X \mid \varphi(x)<a\}$ is open for all $a\in \R$.

\begin{prop}\label{lem:usc}
Let $f \colon X \longrightarrow Y$ be a finite flat morphism between algebraic schemes.
Then the function
\[
e_{f} \colon X \longrightarrow \R, \quad x \mapsto e_{f}(x)
\]
is upper semicontinuous.
\end{prop}
\begin{proof}
Let $\I \subset \O_{X\times_{Y}X}$ be the ideal sheaf of the diagonal.
Let 
\[
n_{0} = \max\{ \dim_{k(y)} \O_{f^{-1}(y)} \mid y \in Y\} 
\]
where $k(y)$ is the residue field of $y$.
Set $\F = {\pr_{1}}_{*}(\O_{X\times_{Y}X}/\I^{n_{0}})$ where $\pr_{1} \colon X\times_{Y}X \longrightarrow X$
is the first projection.
Then 
\begin{align}\label{eq:multcoh}
e_{f}(x) = l_{\O_{X,x}}(\F_{x} {\otimes}_{\O_{X,x}} k(x))
\end{align}
and therefore we are done.

The equation (\ref{eq:multcoh}) follows from the following claim, which is easy to prove.
\begin{claim}
Let $(A,\m) \longrightarrow (B,\n)$ be a local homomorphism between Noetherian local rings 
such that $A/\m$ has characteristic zero, $\m B$ is $\n$-primary, and $[B/\n : A/\m] <\infty$.
Let $I \subset B {\otimes}_{A}B$ be the ideal generated by $b\otimes1-1\otimes b$ for $b\in B$.
If $n \geq l_{B}(B/\m B)$,
then 
\[
l_{B}((B {\otimes}_{A} B)/(I^{n}+\n {\otimes}_{A}B)) = l_{B}(B/\m B).
\]
\end{claim}

\end{proof}

The following theorem says we have control of asymptotic averages of multiplicities
of forward and backward orbits.
These are due to Dinh \cite[Theorem 1.2, Corollary 1.3]{dinh}, Favre \cite[\S 2.5]{fav}, and Gignac \cite[A.3]{gig}.

\begin{thm}\label{thm:gignac}
Let $X$ be an algebraic scheme over a field of characteristic zero.
Let $f \colon X \longrightarrow X$ be a finite flat surjective morphism.

Then 
\begin{enumerate}
\item\label{e+}
The limit 
\begin{align*}
e_{f,+}(x):=e_{+}(x) := \lim_{n\to \infty} e_{f^{n}}(x)^{1/n}
\end{align*}
exists for all $x\in X$.
Moreover, let $L(x) \subset X$ be the $ \omega$-limit set of $x$, i.e.
\begin{align*}
L(x) = \bigcap_{k \geq 0} \overline{ \{ f^{n}(x) \mid n \geq k  \} }.
\end{align*}
Then the generic points $y_{1}, \dots, y_{r}$ of $L(x)$ form an $f$-periodic cycle and we have
\[
e_{+}(x) = (e_{f}(y_{1})\cdots e_{f}(y_{r}))^{1/r}.
\]

For a subvariety $P \subset X$ with generic point $\eta$, we write $e_{f,+}(P) = e_{f,+}(\eta)$.

\item\label{e-}
The limit
\begin{align*}
e_{f,-}(x):=e_{-}(x) := \lim_{n\to \infty} \Bigl( \sup\{e_{f^{n}}(y) \mid y\in X, f^{n}(y)=x\}   \Bigr)^{1/n}
\end{align*}
exists.

\item\label{e-itoe+}
We have
\begin{align*}
e_{-}(x) = \max\left\{ e_{+}(y) \middle|  \txt{$y\in X$ is an $f$-periodic \\scheme point such that $x \in \overline{\{y\}}$}  \right\}.
\end{align*}
Note that this statement includes the existence of $\max$.

\item\label{e-usc}
The function 
\begin{align*}
e_{-} \colon X \longrightarrow \R, \quad x \mapsto e_{-}(x)
\end{align*}
is upper semicontinuous.

\item\label{e-<=e+}
We have
$e_{-}(x) \leq e_{+}(x)$ for all $x \in X$.
If $x$ is $f$-periodic, we have $e_{-}(x)=e_{+}(x)$.

\item\label{supvslim}
Let $Y \subset X$ be a closed subscheme.
Then we have
\begin{align*}
\lim_{n\to \infty} \sup_{y \in Y} \left(\sup \{ e_{f^{n}}(x) \mid f^{n}(x) = y \}^{1/n} \right)= \max\{ e_{f,-}(y) \mid y \in Y \}.
\end{align*}

\end{enumerate}

\end{thm}

\begin{proof}
Let $\tau = \log e_{f} \colon X \longrightarrow \R$.
This is bounded and upper semicontinuous by \cref{lem:usc}.
By \cref{lem:chainrule}, we have 
\[
\tau_{n}(x) := \sum_{k=0}^{n-1}\tau(f^{k}(x)) = \sum_{k=0}^{n-1} \log e_{f}(f^{k}(x)) =\log e_{f^{n}}(x).
\]
Apply \cite[Theorem A.3.1]{gig} or \cite[Theorem E]{gig14} (see also the comments right below them) and \cite[Theorem A.3.5]{gig}
to $f \colon X \longrightarrow X$ and this $\tau$.

For the last statement, the existence of the maximum in the right hand side follows from the upper semicontinuity of $e_{f,-}$.
The equality follows from the next lemma (again apply it to $\tau = \log e_{f}$).
\end{proof}

\begin{rmk}
It is not clear if \cite[Proposition A.3.8]{gig} is true without assuming $f$ is a closed map.
In our case, $f$ is finite and therefore it is closed.
\end{rmk}

\begin{rmk}\label{rmk:twocond}
By \cref{thm:gignac}(3), for $x \in X$ and $ \alpha \in \R$,
the following two statements are equivalent.
\begin{enumerate}
\item For every $f$-periodic subvariety $P \subset X$ such that $x \in P$, we have $e_{f,+}(P) < \alpha$;
\item $e_{f,-}(x) < \alpha$.
\end{enumerate}
\end{rmk}

\begin{lem}
In this lemma, we use the notation in {\rm \cite[A.3]{gig}}.
Let $f \colon X \longrightarrow X$ be a surjective continuous closed self-map of a Zariski topological space $X$ 
(i.e. a Noetherian topological space such that every non-empty irreducible closed subset has a unique generic point).
Let $\tau \colon X \longrightarrow \R$ be a bounded upper semicontinuous function on $X$.
Set 
\begin{itemize}
\item $\tau_{n}(x) = \sum_{k=0}^{n-1} \tau \circ f^{k}(x)$;
\item $\tau_{+}(x) = \lim_{n\to \infty} \tau_{n}(x)/n$;
\item $\tau_{-n}(x) = \sup_{f^{n}(y)=x} \tau_{n}(y)$;
\item $\tau_{-}(x) = \lim_{n\to \infty} \tau_{-n}(x)/n$ 
\end{itemize}
for $x \in X$ as in {\rm \cite[A.3]{gig}}.
Let $Y \subset X$ be a closed subset.
Then we have
\begin{align*}
\lim_{n \to \infty} \sup_{x \in Y} \frac{\tau_{-n}(x)}{n} = \sup_{x \in Y} \tau_{-}(x).
\end{align*}
\end{lem}

\begin{proof}
First note that for any $x \in Y$, we have
\begin{align*}
\liminf_{n \to \infty} \sup_{x' \in Y} \frac{\tau_{-n}(x')}{n} \geq \liminf_{n \to \infty}  \frac{\tau_{-n}(x)}{n} = \tau_{-}(x).
\end{align*}
Thus we get 
\begin{align*}
\liminf_{n \to \infty} \sup_{x \in Y} \frac{\tau_{-n}(x)}{n} \geq \sup_{x \in Y} \tau_{-}(x).
\end{align*}

We prove the reverse inequality (with $\limsup$ instead of $\liminf$ on the left hand side).
We follow the argument in the proof of \cite[Theorem A.3.5]{gig}.

Set
\begin{align*}
c  = \limsup_{n \to \infty} \sup_{x \in Y} \frac{\tau_{-n}(x)}{n}
\end{align*}
and define 
\begin{align*}
Z = \{ y \in X \mid \text{$\tau_{n}(y) \geq cn$ for all $n\geq 1$ }\}.
\end{align*}
(Note that $c < \infty$ since $\tau$ is bounded.) 
Then by \cite[Proposition A.3.7]{gig}, there is a real number $b < c$ and an integer $N \geq 1$ such that
if $n \geq N$ and $x \in X$ satisfies $f^{k}(x) \notin Z$ for $k=0,\dots, n$, then $\tau_{n}(x) \leq bn$. 
Let $Z = Z_{1} \cup \cdots \cup Z_{r}$ be the irreducible decomposition and let $z_{i} \in Z_{i}$ be the generic points.
Let $L(z_{i})$ be the $\omega$-limit set of $z_{i}$ with respect to $f$.

We claim that $L(z_{i}) \cap Y \neq  \emptyset$ for some $i$.
(The following argument also proves $Z$ is non-empty.)
Let us assume $L(z_{i}) \cap Y = \emptyset$ for all $i$ and deduce contradiction.
Choose $s\geq 0$ such that $f^{s}(Z_{i}) \subset L(z_{i})$ for all $i$.
Then by our assumption, we have $f^{-n}(Y) \cap Z =  \emptyset$ for $n \geq s$.
Now for $n \geq N+s$, we have
\begin{align*}
\frac{1}{n}\sup_{x \in Y} \tau_{-n}(x) = \frac{1}{n} \sup \{ \tau_{n}(y) \mid y \in f^{-n}(Y) \}.
\end{align*}
Since $y, f(y), \dots, f^{n-s}(y) \notin Z$ for $y \in f^{-n}(Y)$, we have $\tau_{n-s}(y) \leq b(n-s)$.
As $\tau_{n}(y) = \tau_{n-s}(y) + \tau_{s}(f^{n-s}(y))$, we get
\begin{align*}
\frac{1}{n}\sup_{x \in Y} \tau_{-n}(x) \leq \frac{1}{n} \left( b(n-s) + s \|\tau\|  \right) \to b 
\end{align*}
implying $c \leq b$, a contradiction. Here $\|\tau\| = \sup_{x \in X}|\tau(x)|$. 

Now, take $i$ such that $L(z_{i})\cap Y \neq  \emptyset$.
Let $F \subset L(z_{i})$ be an irreducible component such that $F \cap Y \neq  \emptyset$.
Let $w \in F$ be the generic point.
Note that $w$ is $f$-periodic.
Then by \cite[Comments after Theorem A.3.1]{gig}, we have $\tau_{+}(z_{i}) = \tau_{+}(w)$.
By the definition of $Z$, we have $\tau_{+}(z_{i}) \geq c$.
Take a point $x_{0} \in F \cap Y$.
Then by \cite[Theorem A.3.5.~ (1)]{gig}, we have $\tau_{-}(x_{0}) \geq \tau_{+}(w)$.
Therefore we get
\begin{align*}
c = \limsup_{n \to \infty} \sup_{x \in Y} \frac{\tau_{-n}(x)}{n} \geq \tau_{-}(x_{0}) \geq \tau_{+}(w) = \tau_{+}(z_{i}) \geq c
\end{align*}
and all the inequalities are  actually equalities.
Thus
\begin{align*}
\limsup_{n \to \infty} \sup_{x \in Y} \frac{\tau_{-n}(x)}{n} = \tau_{-}(x_{0}) \leq  \sup_{x \in Y} \tau_{-}(x).
\end{align*}
This finishes the proof.
\end{proof}

\section{Arithmetic distance and orbits}\label{sec:ado}

In this section, we prove \cref{thm:locdistvsht}.

\begin{lem}\label{weakdistrel}
Let $K$ be a field of characteristic zero with proper set of absolute values $M_{K}$.
Let $X, Y$ be nice varieties over $K$ and $f \colon X \longrightarrow Y$ a finite surjective morphism.
For any $y \in Y( \overline{K})$, there exists an $M_{K}$-constant $ \gamma$ such that for all $x \in X( \overline{K})$ and $v \in M( \overline{K})$, 
we have
\[
\delta_{Y,v}(f(x),y) \leq  \sum_{y'\in f^{-1}(y)}e_{f}(y') \delta_{X,v}(x,y') + \gamma_{v}.
\]
\end{lem}
\begin{proof}
Fix $y \in Y(\KK)$.
Up to $M_{K}$-constant, we have
\begin{align*}
\delta_{Y,v}(f(x),y) & = \lambda_{y,v}(f(x)) = \lambda_{f^{-1}(y), v}(x)  & {\small \txt{here $f^{-1}(y)$ is the scheme\\ theoretic inverse image}}\\[3mm]
& \leq \sum_{\tiny\txt{$y' \in X(\KK)$\\ $f(y')=y$}} e_{f}(y')\lambda_{y',v}(x)  & \txt{use \cref{lem:multvslen}}\\[3mm]
& = \sum_{\tiny\txt{$y' \in X(\KK)$\\ $f(y')=y$}} e_{f}(y') \delta_{X,v}(x, y').
\end{align*}

\end{proof}

\begin{rmk}
Note that by the above proof, $M_{K}$-constant $ \gamma$ a priori depends on $y$.
Actually $\gamma$ can be chose to be independent of $y$ (cf. \cite{masil}), but we do not need this fact here.
\end{rmk}

\begin{prop}\label{multbd}
Let $k$ be a field of characteristic zero.
Let $X$ be a nice variety over $k$ and $f \colon X \longrightarrow X$ a surjective morphism.
Let $ \alpha>1$.
Let $y \in X(k)$.
Suppose for every $f$-periodic subvariety $P$ such that $y \in P$, we have $e_{f,+}(P) < \alpha$.
Then we have
\begin{align*}
\lim_{n\to \infty} \frac{\max\{e_{f^{n}}(z) \mid z\in X( \overline{k}), f^{n}(z)=y\}}{ \alpha^{n}}=0.
\end{align*}
\end{prop}
\begin{proof}
Apply \cref{thm:gignac} to $f_{ \overline{k}} \colon X_{ \overline{k}} \longrightarrow X_{ \overline{k}}$.
By \cref{thm:gignac} (\ref{e-}) and (\ref{e-itoe+}), we have
\begin{align*}
&\lim_{n \to \infty} \max\{e_{f^{n}}(z) \mid z\in X( \overline{k}), f^{n}(z)=y\}^{1/n} = e_{f_{ \overline{k}}, -}(y)\\
&=\max\left\{ e_{f_{ \overline{k}}, +}(\xi) \ \middle|\  \txt{$\xi\in X_{ \overline{k}}$ is an $f_{ \overline{k}}$-periodic \\scheme point such that $y \in \overline{\{\xi \}}$}  \right\}.
\end{align*}
Let $\pi \colon X_{ \overline{k}} \longrightarrow X$ be the projection.
By \cref{rmk:multfldext} and the definition of $e_{+}$,  we have $e_{f_{ \overline{k}}, +}(\xi) = e_{f, +}(\pi(\xi))$ for such $\xi$'s.
Since $\pi(\xi)$ is $f$-periodic and $y \in \overline{\{\pi(\xi)\}}$, by our assumption we have $e_{f_{ \overline{k}}, +}(\xi) < \alpha$.
Hence we get 
\begin{align*}
\lim_{n \to \infty} \max\{e_{f^{n}}(z) \mid z\in X( \overline{k}), f^{n}(z)=y\}^{1/n} < \alpha
\end{align*}
and we are done.
\end{proof}

\begin{thm}\label{thm:mainsec5}

Let $K$ be a number field.
Let $X$ be a nice variety over $K$.
Let $f \colon X \longrightarrow X$ be a finite surjective morphism and $S \subset M_{K}$ a finite set.

Let $x, y \in X(K)$ be points satisfying the following:
\begin{enumerate}
\item $ \alpha_{f}(x)>1$;
\item 
For every $f$-periodic subvariety $P \subset X$ such that $y \in P$, we have $e_{f,+}(P) < \alpha_{f}(x)$.
\end{enumerate}
Then 
\[
\lim_{n\to \infty}\frac{\sum_{v\in S} \delta_{X,v}(f^{n}(x),y)}{h_{H}(f^{n}(x))}=0.
\]
\end{thm}

\begin{proof}
Fix a $v \in M_{K}$. 
Fix an algebraic closure $ \overline{K}$ of $K$.
Extend $v$ to an absolute value on $ \overline{K}$ and denote it by $\|\ \|_{v}$.
%Let $K \subset K_{v}$ be the completion and
%$K_{v} \subset \C_{p}$ be the completion of the algebraic closure of $K_{v}$.
%Take the algebraic closure  $\overline{K} \subset \C_{p}$ of $K$ inside $\C_{p}$.
%The absolute value on $\C_{p}$ which extends $v$ is denoted by $\|\ \|_{v}$.
Let us write $\delta(\cdot, \cdot)=\delta_{X, \|\ \|_{v}}(\cdot, \cdot)$, the arithmetic distance function on $X( \overline{K})$
with respect to $\|\ \|_{v}$.
By the triangle inequality (\cite[Proposition 3.1(b)]{sil87}), there is a constant $C_{0}>0$ such that
\[
\min\{ \delta( \alpha, \beta), \delta( \beta, \gamma)\}\leq \delta( \alpha, \gamma) + C_{0}
\]
for all $ \alpha, \beta, \gamma \in X( \overline{K})$.

We may take $h_{H}\geq1$ and may assume that $H$ is very ample.
Fix $x, y\in X(K)$ as in the statement.
Take constants $C_{1}, C_{2}>0$ and $l\geq0$ such that
\[
C_{1}n^{l} \alpha_{f}(x)^{n}\leq h_{H}(f^{n}(x)) \leq C_{2}n^{l} \alpha_{f}(x)^{n}
\]
for all $n\geq1$ (\cref{prop:growthht}).

Let $ \epsilon>0$ be any positive number.
By \cref{multbd}, 
we can find $t( \epsilon)>0$ so that
\[
\frac{\max\{ e_{f^{t( \epsilon)}}(z) \mid f^{t( \epsilon)}(z)=y, z\in X( \overline{K})\}}{ \alpha_{f}(x)^{t( \epsilon)}} < \epsilon.
\]

Let 
\begin{align*}
D( \epsilon) = \max\{ \delta(z,w) \mid \text{$z\neq w$ and $z,w \in f^{-t( \epsilon)}(y) \subset X( \overline{K})$}\}.
\end{align*}
Then for any $ \xi \in X(\overline{K})$, we have
\begin{align*}
\min\{ \delta( \xi, z), \delta( \xi,w) \} \leq \delta(z,w)+C_{0} \leq D( \epsilon)+C_{0}
\end{align*}
for any distinct $z, w \in  f^{-t( \epsilon)}(y)$.
Therefore, there is at most one $z \in  f^{-t( \epsilon)}(y)$ such that $ \delta( \xi, z)> D( \epsilon)+C_{0}$.

By \cref{Roth}(Roth's theorem),
 there is a finite subset $Z( \epsilon) \subset X(K)$ such that
for any $z \in  f^{-t( \epsilon)}(y)$ we have
\begin{align*}
\delta( \xi, z) \leq 3 h_{H}( \xi) \quad \text{for all  $ \xi \in X(K)\setminus Z( \epsilon)$}.
\end{align*}

By \cref{weakdistrel}, there is a constant $C( \epsilon)>0$, which depends on $ \epsilon$ and also $y$, such that
\begin{align*}
\delta(f^{n}(x), y) \leq \sum_{z \in f^{-t( \epsilon)}(y)} e_{f^{t( \epsilon)}}(z) \delta(f^{n-t( \epsilon)}(x), z) +C( \epsilon).
\end{align*}
Let $z_{n} \in f^{-t( \epsilon)}(y)$ be a point such that 
\[
\delta(f^{n-t( \epsilon)}(x), z_{n}) = \max\{ \delta(f^{n-t( \epsilon)}(x), z) \mid z \in  f^{-t( \epsilon)}(y)\}.
\]
Then for all $n\geq t( \epsilon)$, we have
\begin{align*}
&\sum_{z \in f^{-t( \epsilon)}(y)} e_{f^{t( \epsilon)}}(z) \delta(f^{n-t( \epsilon)}(x), z) +C( \epsilon)\\[2mm]
&\leq e_{f^{t( \epsilon)}}(z_{n}) \delta(f^{n-t( \epsilon)}(x), z_{n}) + d^{t( \epsilon)} (D( \epsilon)+C_{0})+C( \epsilon)\\[2mm]
&\leq  e_{f^{t( \epsilon)}}(z_{n}) 3 h_{H}(f^{n-t( \epsilon)}(x))+C_{3}( \epsilon) \qquad \text{if $f^{n-t( \epsilon)}(x) \notin Z( \epsilon)$}
\end{align*}
where $d=\deg f$ and $C_{3}( \epsilon)= d^{t( \epsilon)} (D( \epsilon)+C_{0})+C( \epsilon)$.

Since $O_{f}(x)$ is infinite by assumption, there is a positive integer $n( \epsilon)\geq t( \epsilon)$
such that $f^{n-t( \epsilon)}(x) \notin Z( \epsilon)$ for $n\geq n( \epsilon)$.

Thus for $n\geq n( \epsilon)$, we have
\begin{align*}
\frac{ \delta(f^{n}(x),y)}{h_{H}(f^{n}(x))} &\leq
 \frac{3e_{f^{t( \epsilon)}}(z_{n})h_{H}(f^{n-t( \epsilon)}(x))}{h_{H}(f^{n}(x))}+\frac{C_{3}( \epsilon)}{h_{H}(f^{n}(x))}\\[3mm]
 &\leq    \frac{3e_{f^{t( \epsilon)}}(z_{n}) C_{2}(n-t( \epsilon))^{l} \alpha_{f}(x)^{n-t( \epsilon)}}{C_{1}n^{l} \alpha_{f}(x)^{n}}
 +\frac{C_{3}( \epsilon)}{h_{H}(f^{n}(x))}\\[3mm]
& \leq \frac{3C_{2}}{C_{1}} \frac{e_{f^{t( \epsilon)}}(z_{n})}{ \alpha_{f}(x)^{t( \epsilon)}} +\frac{C_{3}( \epsilon)}{h_{H}(f^{n}(x))}\\[3mm]
& \leq \frac{3C_{2}}{C_{1}} \epsilon +\frac{C_{3}( \epsilon)}{h_{H}(f^{n}(x))}.
\end{align*}

Note that 
\begin{itemize}
\item $3C_{2}/C_{1}$ is independent of $ \epsilon$ and $n$;
\item $C_{3}( \epsilon)$ is independent of $n$ and $h_{H}(f^{n}(x))$ goes to infinity.
\end{itemize}
Thus we are done.

\end{proof}

\begin{ex}
Let $f \colon X \longrightarrow X$ be a finite surjective morphism on a nice variety defined over a number field $K$.
Let $x \in X(K)$ be a point such that $ \alpha_{f}(x) > 1$.
Typical examples of them are:
\begin{itemize}
\item $X = \P^{N}_{K}$, $f$ is non-isomorphic surjection, and $x$ is a point with infinite orbit;
\item $X$ is an abelian variety, $f$ is non-isomorphic surjection, and $x$ is a point with Zariski dense orbit;
\item $X$ is a smooth projective surface, $f$ is an automorphism with dynamical degree larger than one,
and $x$ is a point with Zariski dense orbit. 
\end{itemize}

If there is no periodic subvariety contained in the ramification divisor $R_{f}$, every $y \in X(K)$ satisfies the assumption of the theorem.
This is the case if, for example, $f$ is an automorphism or more generally is \'etale.
In general, by \cref{rmk:basicrmk} (\ref{rmk:onassumption}), there is a non-empty open set $U \subset X$ such that
all $y \in U(K)$ satisfy the assumption of the theorem.
\end{ex}

\begin{ex}
For the morphism $f$ in \cref{ex:P2pcf}, 
every $x \in \P^{2}(K)$ with infinite $f$-orbit and $y \in \P^{2}(K)$ other than $(0:0:1)$
satisfies the assumptions of the theorem.
\end{ex}

Without the assumptions, this theorem has trivial counter examples.

\begin{ex}
Let 
\[
f \colon \P^{N}_{K} \longrightarrow \P^{N}_{K} ; (X_{0}:\cdots :X_{N}) \mapsto (X_{0}^{d}:\cdots :X_{N}^{d})
\]
with $d\geq 2$.
Let 
\begin{align*}
&x=(x_{0}: \cdots : x_{N}) \in \P^{N}(K) \qquad x_{i} \in K;\\
&y=(0:\cdots:0:1) \in \P^{N}(K).
\end{align*}
Then we can take distance function (for fixed $y$) as
\begin{align*}
\delta_{\P^{N}, v}(x , y) = \log\left( \frac{\max\{|x_{0}|_{v},\dots, |x_{N}|_{v}\}}{\max\{|x_{0}|_{v},\dots, |x_{N-1}|_{v}\}}\right)
\end{align*}
for each $v\in M_{K}$.
Suppose 
\[
|x_{N}|_{v} = \max\{|x_{0}|_{v},\dots, |x_{N}|_{v}\} > \max\{|x_{0}|_{v},\dots, |x_{N-1}|_{v}\} =: a
\]
for some $v$.
Then
\begin{align*}
\delta_{\P^{N},v}(f^{n}(x),y) = \log\left(\frac{|x_{N}|_{v}^{d^{n}}}{a^{d^{n}}}\right)=d^{n}\log(|x_{N}|_{v}/a).
\end{align*}
Since $h_{\P^{N}}(f^{n}(x)) = d^{n}h_{\P^{N}}(x)$ where $h_{\P^{N}}$ is the naive height on $\P^{N}$, 
we get 
\begin{align*}
\lim_{n \to \infty}\frac{\delta_{\P^{N},v}(f^{n}(x),y) }{h_{\P^{N}}(f^{n}(x))}=\frac{\log(|x_{N}|_{v}/a)}{h_{\P^{N}}(x)}>0.
\end{align*}

Note that $y$ is contained in the ramification divisor of $f$, which is invariant under $f$ and
$e_{f,+}(y)=d^{N} \geq d = \alpha_{f}(x)$.

\end{ex}

\begin{ex}\label{ex:limnotexist}
Consider again the morphism
\[
f \colon \P^{N}_{K} \longrightarrow \P^{N}_{K} ; (X_{0}:\cdots :X_{N}) \mapsto (X_{0}^{d}:\cdots :X_{N}^{d})
\]
with $d\geq 2$.
Let
\[
y = (0:\zeta_{1}: \cdots : \zeta_{N-1}:1) \in \P^{N}(K)
\]
where $\zeta_{i} \in K$ are $m_{i}$-th primitive root of unity for some $m_{i} \geq 2$.
Suppose $m_{i}$ and $d$ are coprime for all $i$.
Then $y$ is $f$-periodic and easy calculation of the multiplicity shows $e_{f,+}(y) = e_{f,-}(y) = d$.

If we fix $y$ as the above point, the distance function is of the following form up to bounded function for each $v \in M_{K}$:
\begin{align*}
\delta_{\P^{N}, v}(x , y) = \log\left( \frac{\max\{|x_{0}|_{v},\dots, |x_{N}|_{v}\}}{\max\{|x_{0}|_{v}, |x_{1}-\zeta_{1}x_{N}|_{v},\dots, |x_{N-1}-\zeta_{N-1}x_{N}|_{v}\}}\right).
\end{align*}

Now fix a $v \in M_{K}$ and set
\[
x = (x_{0}: \zeta_{1}: \cdots : \zeta_{N-1}:1) \in \P^{N}(K)
\]
where $x_{0} \in K$ and $0 < |x_{0}|_{v} < 1$.
Then for $n\geq 1$
\begin{align*}
\delta_{\P^{N}, v}(f^{n}(x),y) = \log \frac{1}{\max\{ |x_{0}|_{v}^{d^{n}}, |\zeta_{1}^{d^{n}} - \zeta_{1}|_{v}, \dots, |\zeta_{N-1}^{d^{n}} - \zeta_{N-1}|_{v} \}}.
\end{align*}
Hence we get
\begin{align*}
&\limsup_{n \to \infty} \frac{\delta_{\P^{N}, v}(f^{n}(x),y) }{h_{\P^{N}}(f^{n}(x))} = \frac{-\log |x_{0}|_{v}}{h_{\P^{N}}(x)} > 0\\
&\liminf_{n \to \infty} \frac{\delta_{\P^{N}, v}(f^{n}(x),y) }{h_{\P^{N}}(f^{n}(x))} = 0 \quad \text{if $d \not\equiv 1\mod m_{i}$ for at least one $i$.}
\end{align*}

Note that in this case $x$ cannot have Zariski dense $f$-orbit if $N \geq 2$.

\end{ex}

\begin{ex}
Let 
\[
f \colon \P^{1}_{\Q} \longrightarrow \P^{1}_{\Q}; (X:Y) \mapsto (X^{d}:Y^{d})
\]
where $d \geq 2$.
Let 
\begin{align*}
&x = (a/b : 1)\\
&y = (0:1)
\end{align*}
where $a, b \in \Z$ are non zero coprime integers.
Let $v \in M_{\Q}$.
If $|a/b|_{v} < 1$, then as in the previous example, we can calculate
\[
\lim_{n \to \infty} \frac{\delta_{\P^{1}, v}(f^{n}(x),y)}{h_{\P^{1}}(f^{n}(x))} = - \frac{\log |a/b|_{v}}{h_{\P^{1}}(x)} = - \frac{\log |a/b|_{v}}{\log \max\{|a|, |b|\}} .
\]
When $a, b$ varies, these values form a dense subset of the interval $[0,1]$.
\end{ex}

\begin{ex}
Let 
\[
f \colon \P^{2}_{\Q} \longrightarrow \P^{2}_{\Q} ; (X : Y : Z) \mapsto (X-2Y : X+4Y : 3X + 7Y + 5Z).
\]
Then $ \alpha_{f}(x) = 1$ for all $x \in \P^{2}(\Q)$ and thus there are no points that satisfy the assumption.
Let 
\begin{align*}
&x=(5:-2:0) \in \P^{2}(\Q)\\
&y=(0:0:1) \in \P^{2}(\Q).
\end{align*}
Then we can calculate $f^{n}(x)=(6\cdot 2^{n} - 3^{n}: -3 \cdot 2^{n} + 3^{n} : 2^{n}-2\cdot 3^{n} + 5^{n})$.
(Diagonalize the defining matrix of $f$. We can easily see that this orbit is Zariski dense.)
Thus we get
\begin{align*}
\delta_{\P^{2}, \infty}(f^{n}(x), y) =  n\log \frac{5}{3} + O(1) \quad \text{as $n \to \infty$}
\end{align*}
where $ \delta_{\P^{2}, \infty}$ is the arithmetic distance function at the infinite place of $\Q$.
On the other hand, it is easy to see that $\gcd(6\cdot 2^{n} - 3^{n}, -3 \cdot 2^{n} + 3^{n} , 2^{n}-2\cdot 3^{n} + 5^{n})=1$ and
hence
\[
h_{\P^{2}}(f^{n}(x)) =\log \max\{|6\cdot 2^{n} - 3^{n}|, |-3 \cdot 2^{n} + 3^{n}| , |2^{n}-2\cdot 3^{n} + 5^{n}|  \}= n\log 5 + O(1).
\]
Therefore
\begin{align*}
\lim_{n \to \infty}\frac{ \delta_{\P^{2},\infty}(f^{n}(x),y)}{h_{\P^{2}}(f^{n}(x))} = 1- \frac{\log 3}{\log 5}.
\end{align*}

\end{ex}

%%%%%%%%%%%%%%%%
\if0

\begin{ex}
Let 
\[
f \colon \P^{2}_{\Q} \longrightarrow \P^{2}_{\Q} ; (X : Y : Z) \mapsto (X : Y : Z+X).
\]
Let 
\begin{align*}
&x=(1:1:0) \in \P^{2}(\Q)\\
&y=(0:0:1) \in \P^{2}(\Q).
\end{align*}
Then $f^{n}(x)=(1:1:n)$.
Thus we get
\begin{align*}
\delta_{\P^{2}, \infty}(f^{n}(x), y) = \log n + O(1) \quad \text{as $n \to \infty$}
\end{align*}
where $ \delta_{\P^{2}, \infty}$ is the arithmetic distance function at the infinite place of $\Q$.
Since $h_{\P^{2}}(f^{n}(x))=h_{\P^{2}}(1:1:n)=\log n$, we have
\begin{align*}
\lim_{n\to \infty} \frac{ \delta_{\P^{2},\infty}(f^{n}(x),y)}{h_{\P^{2}}(f^{n}(x))} =1.
\end{align*}
Note that $ \alpha_{f}(x)=1$.
\end{ex}
\fi
%%%%%%%%%%%%%%%%%

\begin{ex}
Let 
\[
f \colon \P^{2}_{\Q} \longrightarrow \P^{2}_{\Q} ; (X : Y : Z) \mapsto (X^{3} : Y^{3}+YZ^{2} : Z^{3}).
\]
Let 
\[
x=(1:0:2), \ y=(0:0:1).
\]
Then $f(y)=y$ and 
\[
e_{f, +}(y)=3 = d_{1}(f)= \alpha_{f}(x).
\]
It is easy to check that $ \delta_{\P^{2}, \infty}(f^{n}(x),y)/h_{\P^{2}}(f^{n}(x))$ converges to $1$.
Note that in this example, $O_{f}(x)$ is infinite but not Zariski dense.
\end{ex}

\begin{que}
Do there exist $f, X, S, x$, and $y$ which satisfy the following?:
\begin{enumerate}
\item $X$ is a nice variety over a number field $K$ with $\dim X\geq 2$;
\item $f \colon X \longrightarrow X$ is a surjective morphism;
\item $x, y \in X(K)$ such that $e_{f,-}(y) = \alpha_{f}(x) >1$ (or at least $e_{f,-}(y) < \deg f$);
\item $O_{f}(x)$ is Zariski dense;
\item $S \subset M_{K}$ is a finite set and 
\[
\frac{ \sum_{v \in S}\delta_{X,v}(f^{n}(x),y)}{h_{H}(f^{n}(x))}
\]
does not converge to $0$ (for any ample height $h_{H}$).
\end{enumerate}

Is it possible to construct such examples on projective spaces?
\end{que}

\begin{que}
When the assumption (2) in \cref{thm:mainsec5} is not satisfied, what can we say about the sequence
\[
\frac{ \sum_{v \in S}\delta_{X,v}(f^{n}(x),y)}{h_{H}(f^{n}(x))} \quad (n \geq 1) \quad?
\]
It can happen that it does not converge (cf.\ \cref{ex:limnotexist}. In the example, the sequence have at most two accumulation points).
Can the sequence have infinitely many accumulation points?

\end{que}

\section{Size of coordinates of orbits}\label{sec:sco}

\begin{prop}\label{prop:sizeweak}
Let $f \colon \P^{N}_{\Q} \longrightarrow \P^{N}_{\Q}$ be a surjective endomorphism.
For each $i=0,\dots, N$, let $y_{i} = (0:\cdots: 0:1:0: \cdots :0) \in \P^{N}$ be the points such that only the $i$-th coordinate is non-zero.
Let $x \in \P^{N}( \Q)$ be a point and write $f^{n}(x)=(a_{0}(n):\cdots:a_{N}(n))$
where $a_{0}(n),\dots,a_{N}(n)$ are coprime integers.

Suppose 
\begin{enumerate}
\item $O_{f}(x)$ is infinite;
\item For every $f$-periodic subvariety $P$  such that $y_{i} \in P$ for some $i$, we have $e_{f,+}(P) < d_{1}(f)$.
\end{enumerate}
Then we have
\begin{align*}
\lim_{n\to \infty} \frac{\log\left(\max_{0\leq j \leq N, j\neq i}\{|a_{j}(n)| \}\right)}{\log\left(\max_{0\leq j \leq N}\{|a_{j}(n)| \}\right)}=1.
\end{align*}
for all $i$.
\end{prop}

\begin{proof}
Let $h$ be the naive height function on $\P^{N}_{\Q}$ and
$ \delta = \delta_{\P^{N}_{\Q}, \infty}$ be an arithmetic distance function on $\P^{N}_{\Q}$ at the infinite place.
We can write
\begin{align*}
h(f^{n}(x)) & = \log\left(\max_{0\leq j \leq N}\{|a_{j}(n)|\}\right)\\[3mm]
\delta(f^{n}(x), y_{i}) & =\log\left( \frac{\max_{0\leq j \leq N}\{|a_{j}(n)|\}}{\max_{0\leq j \leq N, j\neq i}\{|a_{j}(n)|\}} \right).
\end{align*}

By the assumption and \cref{thm:locdistvsht}, we have
\begin{align*}
0 & = \lim_{n\to \infty} \frac{\delta(f^{n}(x), y_{i})}{h(f^{n}(x))}\\[3mm]
& =\lim_{n\to \infty} \frac{\log\left(\max_{0\leq j \leq N}\{|a_{j}(n)|\}\right)-\log\left(\max_{0\leq j \leq N, j\neq i}\{|a_{j}(n)|\}\right)}{\log\left(\max_{0\leq j \leq N}\{|a_{j}(n)|\}\right)}\\[3mm]
& = 1- \lim_{n\to \infty} \frac{\log\left(\max_{0\leq j \leq N, j\neq i}\{|a_{j}(n)| \}\right)}{\log\left(\max_{0\leq j \leq N}\{|a_{j}(n)| \}\right)}.
\end{align*}

\end{proof}

\begin{ex}\label{ex:cetodls}
Let 
\begin{align*}
&g \colon \P^{2}_{\Q} \longrightarrow \P^{2}_{\Q} ; (X:Y:Z) \mapsto (X^{3}: Y^{3} :Z^{3}) ;\\
&\sigma \colon \P^{2}_{\Q} \longrightarrow \P^{2}_{\Q} ; (X:Y:Z) \mapsto (X+Y+Z : 2X+Y+Z : X-Y+Z).
\end{align*}
and define $f= \sigma^{-1}\circ g \circ \sigma$.
Let $x\in \P^{2}_{\Q}$ be any point such that $ \sigma(x) \in \{XYZ \neq 0\}$ and has infinite $f$-orbit.
Let
\[
 y_{0}=(1:0:0),\  y_{1}=(0:1:0),\  y_{2}=(0:0:1).
 \]
Then $f, x, y_{i}$ satisfies all assumptions in \cref{thm:locdistvsht}.

If we set $x=( 2:3 :-4 )$, then $ \sigma(x)=(1:3:-5)$ and $x$ has infinite $f$-orbit.
Write $f^{n}(x)=(a(n):b(n):c(n))$
where $a(n), b(n), c(n)$ are integers with $\gcd=1$.
Then numerical calculation shows:

\begin{table}[htb]
	\begin{tabular}{c||c|c|c}
	$n$ & $\log|a(n)|$ & $\log|b(n)|$ & $\log|c(n)|$ \\ \hline
	1 & 3.25809653802148& 4.14313472639153&4.47733681447821\\
	2 & 9.88745979145893& 13.7917945433468& 13.8117474864837\\
	3 & 29.6625317940388& 42.7616764551608& 42.7616785021394\\
	4 & 88.9875953821169& 129.671323726602& 129.671323726602\\
	5 & 266.962786146351& 390.400265540926& 390.400265540926\\
	6 &800.888358439052& 1172.58709098390& 1172.58709098390
	\end{tabular}
\end{table}

This does not contradict to \cref{prop:sizeweak}, 
but this looks as if it gives a counter example to \cite[Conjecture 21.4]{adsurv}, and it actually does.
Indeed, let $H_{0}$ be the hyperplane defined by $X=0$ and $h$ be the naive height on $\P^{2}$, then 
\begin{align*}
\frac{ \lambda_{H_{0}, \infty}(f^{n}(x))}{h(f^{n}(x))} & = \frac{ \lambda_{H_{0}, \infty}(\sigma^{-1}g^{n}\sigma(x))}{h(\sigma^{-1}g^{n}\sigma(x))} \\[3mm]
&= \frac{ \lambda_{\sigma(H_{0}), \infty}(g^{n}(\sigma(x)))}{h(g^{n}(\sigma(x)))} + o(1)\\[3mm]
&= 1- \frac{\log |-a'(n)+b'(n)|}{\log \max\{|a'(n)|, |b'(n)|, |c'(n)|\}} + o(1)\\[3mm]
& = 1 - \frac{\log(3^{3^{n}}-1)}{\log 5^{3^{n}}} + o(1)\\[3mm]
&\xrightarrow{n \to \infty}  1- \frac{\log 3}{\log 5}
\end{align*}
where we write $g^{n}(\sigma(x))=(a'(n): b'(n): c'(n))$.
This shows that
\[
\lim_{n\to \infty} \frac{\log |a(n)|}{\log \max\{|a(n)|, |b(n)|, |c(n)|\}} = \frac{\log 3}{\log 5}.
\]
On the other hand, the first coordinates of $f^{n}$ for $n=1,2,3$ are not constant multiples of powers of $X$.

Finally note that this $f$ does not satisfy the assumption of \cref{cor:dls}.
The point $\sigma^{-1}(0:0:1)=(0:1:-1) \in H_{0}$ is an $f$-periodic point such that
$e_{f,+}(0:1:-1)=9>d_{1}(f)=3$.
\end{ex}

\section{Local height associated with subschemes}\label{sec:lochtdiv}

In this section we consider \cref{mainQ}(2) for $Y$ being an arbitrary proper closed subscheme assuming Vojta's conjecture.
To bound the growth of the local height associated to a closed subscheme $Y$, 
we want to apply Vojta's conjecture to the scheme theoretic inverse images $(f^{n})^{-1}(Y)$.
But in general, $(f^{n})^{-1}(Y)$ could have bad singularities.
To overcome this problem, we estimate the asymptotic badness of the singularities of $(f^{n})^{-1}(Y)$ (\S \ref{subsec:singpull})
and also reformulate Vojta's conjecture in a slightly different way (\S \ref{subsec:vojta}).

\subsection{Singularities of pull-backs}\label{subsec:singpull}
In this subsection, we work over a field $k$ of characteristic zero.

\begin{defn}
Let $X$ be a variety over $k$.
For a proper closed subscheme $Y \subset X$ and a (scheme) point $x \in X$,
the  \emph{multiplicity of $Y$ at $x$} is defined as
\[
\mult_{x}Y := \mult_{x}\I_{Y} := \max\{ m \mid (\I_{Y})_{x} \subset \mathfrak{m}_{x}^{m}\}
\]
where $\I_{Y}$ is the ideal sheaf defining $Y$.
Note that since $Y \neq X$ and $X$ is a variety, $(\I_{Y})_{x} \neq 0$ for any $x \in X$
and $\mult_{x}Y$ is finite. 
\end{defn}

\begin{rmk}
When $X$ is smooth,
the function $X \longrightarrow \Z,\ x \mapsto \mult_{x}Y$ is upper semicontinuous. 
Thus, in particular,
\[
\max_{\tiny\txt{$x\in X$\\closed point}} \{\mult_{x}Y\} = \max_{\tiny\txt{$x\in X$\\scheme point}} \{\mult_{x}Y\}.
\]
Here is a sketch of how to deduce the upper semicontinuity from that of Samuel multiplicity.
We may assume $X$ is affine, say $X = \Spec A$ and $Y$ is a closed subscheme defined by
a non-zero ideal $I \subset A$.
Let $a_{1}, \dots , a_{r} \in I$ be a generator of $I$.
Then
\[
\mult_{x}Y = \mult_{x} I =\min_{i} \{ \mult_{x}a_{i}A \}
\]
for any $x \in X$.
Thus we may assume that $I$ is a principal ideal.
Then, since $X$ is assumed to be smooth, $\mult_{x}I$ is equal to the Samuel multiplicity of the local ring $\O_{X,x}/I\O_{X,x}$
when $x \in Y$.
Hence, it is enough to show that the function 
\[
\Spec A/I \longrightarrow \Z, \p \mapsto \text{Samuel multiplicity of $(A/I)_{\p}$}
\]
is upper semicontinuous.
Noticing that all irreducible components of $\Spec A/I$ have the same dimension,
this follows from, for example, \cite[Chapter III Remark (1.3), Chapter III (2.4)]{ben}.

\end{rmk}

\begin{rmk}
Let $k \subset k'$ be a field extension.
Let $p \colon X_{k'} \longrightarrow X$ be the projection.
Then $\mult_{x'}Y_{k'} = \mult_{p(x')}Y$ for any $x' \in X'$.
(cf. \cref{claim:flatmult}.)
\end{rmk}

\begin{lem}\label{lem:multbd}
Let $f \colon X \longrightarrow Z$ be a finite flat morphism between varieties.
Let $Y\subset Z$ be a proper closed subscheme.
Let $x \in X$ be a scheme point.
Then we have
\[
\mult_{x}f^{-1}(Y) < e_{f}(x) (\mult_{f(x)}Y+1).
\]
\end{lem}
\begin{proof}
This follows from the following claim.
\end{proof}

\begin{claim}\label{claim:flatmult}
Let $(A, \m) \longrightarrow (B, \n)$ be a flat local homomorphism between Noetherian local rings such that
$\m B$ is $\n$-primary.
Let $\a \subset A$ be a non-zero ideal.
Let
\[ 
e = l_{B}(B/\m B) \quad \text{and} \quad m_{0}= \max\{m\mid \a \subset \m^{m}\}.
\]
Then we have
\[
\a B \nsubset \n^{e(m_{0}+1)}.
\]
\end{claim}
\begin{proof}
Since $\n^{e} \subset \m B$, we have $\n^{e(m_{0}+1)} \subset \m^{m_{0}+1}B$.
Since the map $\a \longrightarrow A \longrightarrow A/\m^{m_{0}+1}$ is not zero and
$A \longrightarrow B$ is faithfully flat, 
\[
\a {\otimes}_{A}B \longrightarrow B \longrightarrow B/\m^{m_{0}+1}B
\]
is also not zero. This implies $\a B \nsubset \m^{m_{0}+1}B$.
\end{proof}

Next, we review the definitions of multiplier ideals and log canonical thresholds. 
See for example \cite[Chapter 9]{laz2} for the basic properties of them.

\begin{defn}\label{def:multiplierideal}
Suppose $k$ is algebraically closed.
Let $X$ be a smooth variety over $k$ and $\I \subset \O_{X}$ be a non-zero ideal sheaf.
Let $c \in \Q$ be a positive rational number.
Then the \emph{multiplier ideal sheaf $\J(X, \I^{c})$ associated with $\I$ and $c$} is defined as follows.  
Let $\pi \colon X' \longrightarrow X$ be a log resolution of $(X, \I)$, i.e. $\pi$ is a projective birational morphism such that
\begin{itemize}
\item $X'$ is a smooth variety over $k$;
\item $\I \O_{X'} = \O_{X'}(-F)$ for some effective Cartier divisor $F$ on $X$;
\item $\Exc(\pi) \cup \Supp F$ is simple normal crossing.
Here $\Exc(\pi)$ is the exceptional locus of $\pi$, that is, $\Exc(\pi) = X' \setminus \pi^{-1}(U)$
where $U \subset X$ is the largest open subset such that $\pi^{-1}(U) \xrightarrow{\pi} U$ is isomorphic.
\end{itemize}
For any canonical divisor $K_{X}$ of $X$, there is a unique canonical divisor $K_{X'}$ on $X'$ such that
$\pi_{*}K_{X'}=K_{X}$. Then $K_{X'/X} := K_{X'}-\pi^{*}K_{X}$ is an effective exceptional divisor which 
is independent of the choice of $K_{X}$ and define
\begin{align*}
\J(X, \I^{c}) = \pi_{*}\O_{X'}(K_{X'/X} - \floor*{cF})
\end{align*}
where $\floor*{cF}$ is the divisor obtained by taking round down of coefficients of $cF$.
Note that since $\pi_{*}\O_{X'}(K_{X'/X})=\O_{X}$, this is an ideal sheaf on $X$.
See for example \cite[9.2]{laz2} for the independence of $\J(X, \I^{c})$ on $\pi$ and basic properties.
\end{defn}

\begin{defn}
Notation as in \cref{def:multiplierideal}.
Let $x \in X$ be a closed point.
The  \emph{log canonical threshold} of $\I$ at $x$ is 
\begin{align*}
\lct(\I ; x) = \inf \{ c \in \Q_{>0} \mid \J(X, \I^{c})_{x} \subset \m_{x} \}.
\end{align*}
The  \emph{log canonical threshold} of $(X, \I)$ is 
\begin{align*}
\lct(X, \I) = \min\{ \lct(\I; x) \mid \text{$x \in X$ closed point}\}.
\end{align*}
These are positive rational numbers.
\end{defn}

\begin{rmk}
Notation as in \cref{def:multiplierideal}.
For a rational number $c >0$, the following are equivalent:
\begin{enumerate}
\item $c \leq \lct(X, \I)$;
\item all coefficients of the divisor $K_{X'/X} - cF$ are $\geq -1$.
\end{enumerate}
\end{rmk}

\begin{ex}
Let $D$ be a simple normal crossing divisor on a smooth variety $X$.
Then $\lct(X, \O_{X}(-D))=1$. 
In general, smaller lct corresponds to worse singularities.
\end{ex}

We make the following possibly non standard definition.
\begin{defn}
Let $X$ be a nice variety over $k$.
Let $Y \subset X$ a proper closed subscheme.
Then we define  \emph{the log canonical threshold of $(X, Y)$} as 
\[
\lct(X, Y):= \lct(X_{ \overline{k}}, Y_{ \overline{k}}) := \lct(X_{ \overline{k}}, \I_{ Y_{ \overline{k}}}).
\]
\end{defn}

By the following, we can connect the multiplicity and lct.
\begin{lem}\label{lem:lctvsmult}
Let $X$ be a nice variety over $k$.
Let $Y\subset X$ be a proper closed subscheme. 
Then
\[
\lct(X,Y) \geq \frac{1}{\max_{y\in Y} \{\mult_{y}Y\}}.
\]
\end{lem}
\begin{proof}
This follows from the following claim.
(When $Y$ is a divisor, the statement follows from \cite[Proposition 9.5.13]{laz2})
\end{proof}

\begin{claim}
Let $X$ be a smooth variety over an algebraically closed field of characteristic zero.
Let $\a \subset \O_{X}$ be a non zero ideal sheaf.
Let $x\in X$ be a closed point. For any $c \in \Q$ such that $0<c<(\mult_{x}\a)^{-1}$,
we have $ \mathcal{J}(X,  \a^{c})_{x} = \O_{X,x}$.
\end{claim}
\begin{proof}
This follows from the restriction theorem on multiplier ideals \cite[Example 9.5.4]{laz2}
 and induction on dimension (note that if $\dim X=1$, this is trivial.).
Note that for any point $x \in X$ and general hypersurface $H$ through $x$ (defined around $x$),
we have $\mult_{x}\a = \mult_{x}\a_{H}$, where $\a_{H}$ is the image of $\a$ (i.e. $\a_{H}= \a \O_{H}$).
Indeed, the inequality $\mult_{x}\a \leq \mult_{x}\a_{H}$ is trivial.
For the reverse inequality, let $n = \mult_{x}\a$ and pick $a \in \a \setminus \m_{x}^{n+1}$, where
$\m_{x}$ is the maximal ideal of the local ring $\O_{X, x}$ at $x$.
Let $\varphi \in \m_{x} \setminus \m_{x}^{2}$ be such that $\varphi \mod \m_{x}^{2}$  
does not divide $a \mod \m_{x}^{n+1}$ in the polynomial ring $\bigoplus_{i\geq 0} \m_{x}^{i}/\m_{x}^{i+1}$.
Note that such a $\varphi$ is general in $\m_{x}/\m_{x}^{2}$.
Then the image of $a$ in $\O_{X, x}/\varphi \O_{X,x}$ is not contained in $(\m_{x}/\varphi \O_{X,x})^{n+1}$. 
If we take $H$ to be the hypersurface defined by $\varphi$, we get $\mult_{x}\a \geq \mult_{x}\a_{H}$.

\end{proof}

\begin{cor}\label{lem:lctofpull}
Let $X$ be a nice variety over $k$.
Let $f \colon X \longrightarrow X$ be a surjective morphism.
Let $Y\subset X$ be a proper closed subscheme.
Then we have
\[
\liminf_{n\to \infty} \lct(X, (f^{n})^{-1}(Y))^{1/n} \geq \frac{1}{\max\{ e_{f,-}(y) \mid y \in Y\}}.
\]
\end{cor}
\begin{proof}
By \cref{lem:multbd,lem:lctvsmult}, we have
\begin{align*}
\lct(X, (f^{n})^{-1}(Y)) &\geq \frac{1}{\max_{x\in  (f^{n})^{-1}(Y)} \{\mult_{x}(f^{n})^{-1}(Y)\}}\\
&\geq \frac{1}{\max_{x\in  (f^{n})^{-1}(Y)}\{ e_{f^{n}}(x) (\mult_{f^{n}(x)}Y +1)\}}.
\end{align*}
Let $m_{0} = \max\{ \mult_{x}Y \mid x\in Y\}+1$.
Then 
\begin{align*}
\lct(X, (f^{n})^{-1}(Y)) &\geq \frac{1}{m_{0} \max_{x\in  (f^{n})^{-1}(Y)}\{e_{f^{n}}(x)\}}\\
&\geq \frac{1}{m_{0} \sup_{y\in Y} \max\{ e_{f^{n}}(x) \mid f^{n}(x)=y\}}.
\end{align*}
By \cref{thm:gignac}(\ref{supvslim}), we get
\[
\liminf_{n\to \infty} \lct(X, (f^{n})^{-1}(Y))^{1/n} \geq \frac{1}{\max\{ e_{f,-}(y) \mid y \in Y\}}.
\]
\end{proof}

\subsection{Vojta's conjecture}\label{subsec:vojta}

We reformulate Vojta's conjecture using log canonical thresholds.
Let us first recall the usual Vojta's conjecture.

\begin{conj}[{Vojta's conjecture \cite[Conjecture 3.4.3]{vojta87}}]
Let $K$ be a number field, $X$ a nice variety over $K$, $D$ an effective simple normal crossing divisor on $X$,
$H$ a big divisor on $X$, and $K_{X}$ a canonical divisor on $X$.
Fix a local height function $ \{\lambda_{D,v}\}_{v \in M_{K}}$ and global height functions $h_{H}$ and $h_{K_{X}}$.
Then for any finite set $S \subset M_{K}$ and any positive number $ \epsilon>0$,
there is a proper closed subset $Z \subset X$ such that
\[
\sum_{v \in S} \lambda_{D,v}(x) \leq \epsilon h_{H}(x) - h_{K_{X}}(x)
\]
for all $x \in (X\setminus Z)(K)$.
\end{conj}

\begin{prop}\label{prop:vojta}
Let $K$ be a number field and let $S \subset M_{K}$ be a finite set.
Let $X$ be a nice variety over $K$.
Let $Y \subset X$ be a proper closed subscheme and let $\{ \lambda_{Y,v} \}_{v \in M_{K}}$ be a local height function associated with $Y$.
Let $h_{H}, h_{K_{X}}$ be height functions associated with an ample divisor $H$ and a canonical divisor $K_{X}$ on $X$.

Let $c>0$ be a positive rational number such that the pair $(X, \I_{Y}^{c})$ is log canonical.
Let $ \epsilon>0$ be any positive number.
Then Vojta's conjecture (for blow ups of $X$) implies that there is a proper closed subset $Z \subset X$ such that
\[
c\sum_{v\in S}\lambda_{Y,v}(x) \leq \epsilon h_{H}(x)- h_{K_{X}}(x)
\]
for all  $x\in (X\setminus Z)(K)$

In particular, we can take $c=\lct(X, \I_{Y})$.
\end{prop}
\begin{proof}
Let $\mu \colon X' \longrightarrow X $ be a log resolution of $(X, Y)$.
Write 
\begin{align*}
K_{X'/X} - c\mu^{-1}(Y) = \sum_{i}a_{i}E_{i}
\end{align*}
where $E_{i}$ are distinct prime divisors on $X'$.
(Note that the scheme theoretic inverse image $\mu^{-1}(Y)$ is an effective Cartier divisor on $X'$.)
Note that $a_{i} \geq -1$ since $(X, \I_{Y}^{c})$ is log canonical and $K_{X'/X}$ is (exceptional) effective since $X$ is smooth.
(We can do this at least over a finite extension of $K$. But it is enough to show the statement after taking base change to a finite extension of $K$.)

Let $ \epsilon>0$ be a positive number.
Since $c\mu^{-1}(Y)-K_{X'/X}$ is a divisor with simple normal crossing support with coefficient $\leq 1$, by Vojta's conjecture,
there is a proper closed subset $Z'\subset X'$ such that
\[
c\sum_{v \in S} \lambda_{\mu^{-1}(Y),v} - \sum_{v\in S} \lambda_{K_{X'/X},v} + h_{K_{X'}} \leq \epsilon h_{\mu^{*}H}
\]
on $(X'\setminus Z')(K)$.
Since we can take $K_{X'}$ so that $K_{X'}= \mu^{*}K_{X} + K_{X'/X}$, we get 
\[
c\sum_{v \in S} \lambda_{Y,v}\circ \mu + h_{K_{X}}\circ \mu +\sum_{v \notin S} \lambda_{K_{X'/X},v} \leq \epsilon h_{H}\circ \mu
\]
on $(X'\setminus Z')(K)$ (enlarge $Z'$ if necessary).
Thus the statement holds for $Z = \mu(Z')  \cup \mu(\Exc(\mu))$
(if we chose $ \lambda_{K_{X/X'},v} \geq 0$ outside $\Supp K_{X/X'}$).
\end{proof}

\begin{rmk}
Vojta \cite{vojta12} and Yasuda \cite{yasd} formulated versions of Vojta's conjecture for possibly singular subschemes 
or singular ambient varieties using multiplier ideals.
They also proved equivalences of the new conjectures with original Vojta's conjecture.
The idea of the proof of \cref{prop:vojta} is more or less along the same line with the one in \cite{vojta12},
but it does not seem this proposition directly follows from for example \cite[Conjecture 4.2]{vojta12}.
\end{rmk}

\subsection{Proof of \cref{thm:divlcht,thm:divlchtnongen}}

Let $K$ be a number field.

\begin{lem}\label{lem:keycalc}
Let $X$ be a nice variety over $K$.
Let $f \colon X \longrightarrow X$ be a finite surjective morphism and let $S \subset M_{K}$ be a finite set.
Let $Y\subset X$ be a proper closed subscheme and $x \in X(K)$ a point.
Fix an ample height $h_{H}$ on $X$, a height $h_{K_{X}}$ associated with the canonical divisor $K_{X}$, and
a local height function $ \{\lambda_{Y,v}\}_{v\in M_{K}}$ associated with $Y$.

Set $e = \max\{e_{f,-}(x) \mid x \in Y\}$.

Assume Vojta's conjecture (for blow ups of $X$).
Then for any $ \epsilon >0$, there exists a positive integer $t_{0}\geq1$
such that the following holds.
For any $t \geq t_{0}$ and for any $ \epsilon'>0$, there exists a proper Zariski closed subset 
$Z=Z(Y, \epsilon, t, \epsilon') \subset X$ and
a constant $C(t)\geq 0$ such that for all $n\geq t$, we have
\begin{align*}
\sum_{v \in S} \lambda_{Y,v}(f^{n}(x)) \leq (e+ \epsilon)^{t} ( \epsilon'h_{H}(f^{n-t}(x)) - h_{K_{X}}(f^{n-t}(x))) + C(t)
\end{align*} 
if $f^{n-t}(x) \notin Z$. Note that $Z, C(t)$ are independent of $x$ and $n$.
\end{lem}

\begin{proof}
By \cref{lem:lctofpull}, for any $ \epsilon>0$, there is a $t_{0}\geq 1$ such that
\[
\lct(X, (f^{t})^{-1}(Y)) \geq \frac{1}{(e+ \epsilon)^{t}}.
\]
for all $t \geq t_{0}$
Therefore, we can use $1/(e+ \epsilon)^{t}$ as $c$ in \cref{prop:vojta}.
Fix a $t\geq t_{0}$.
By \cref{prop:vojta}, for any $ \epsilon'>0$, 
there exist a proper Zariski closed subset $Z \subset X$ and
a constant $C(t)\geq 0$ such that
\begin{align*}
&\sum_{v\in S} \lambda_{Y,v}(f^{n}(x)) \leq \sum_{v\in S} \lambda_{(f^{t})^{-1}Y,v}(f^{n-t}(x)) + C(t)  
& {\small \txt{functoriality of \\local heights}}\\[3mm]
&\leq (e+ \epsilon)^{t} ( \epsilon' h_{H}(f^{n-t}(x)) - h_{K_{X}}(f^{n-t}(x)) ) + C(t) 
& {\small \text{ \cref{prop:vojta}}}
\end{align*}
if $f^{n-t}(x) \notin Z$.
\end{proof}

\begin{proof}[Proof of \cref{thm:divlcht}]
We may take $h_{H}\geq1$ without loss of generality.

We use the notation in \cref{lem:keycalc}.
By the assumption and \cref{thm:gignac}, we have $e < \alpha_{f}(x)$.
Let $ \epsilon_{0}>0$ be any positive number.
Let $ \epsilon>0$ be such that $e+ \epsilon < \alpha_{f}(x)$.
For this $ \epsilon$, apply \cref{lem:keycalc} and take $t_{0}$.
Take a $t\geq t_{0}$ so that 
\[
\frac{(e+ \epsilon)^{t}}{ \alpha_{f}(x)^{t}} < \epsilon_{0}.
\]
Then again by \cref{lem:keycalc} (applied to $ \epsilon'=1$), there is a proper closed subset $Z \subset X$ such that
\begin{align*}
\sum_{v\in S}\lambda_{Y,v}(f^{n}(x)) \leq  \epsilon_{0} \alpha_{f}(x)^{t}( h_{H}(f^{n-t}(x)) - h_{K_{X}}(f^{n-t}(x))) + C(t)
\end{align*}
for all $n\geq t$ such that $f^{n-t}(x) \notin Z$.
Since we assume $O_{f}(x)$ is generic, there is $n_{0}$ such that for all $n\geq n_{0}$,
we have $f^{n-t}(x) \notin Z$.

Now take constants $C_{1}, C_{2}>0$ and $l\geq0$ such that
\[
C_{1}n^{l} \alpha_{f}(x)^{n}\leq h_{H}(f^{n}(x)) \leq C_{2}n^{l} \alpha_{f}(x)^{n}
\]
for all $n\geq1$ (\cref{prop:growthht}).
Also, take a constant $C_{3}>0$ such that $h_{H}-h_{K_{X}} \leq C_{3}h_{H}$.
Then we get for all $n\geq n_{0}$
\begin{align*}
\frac{\sum_{v\in S}\lambda_{Y,v}(f^{n}(x))}{h_{H}(f^{n}(x)) } &\leq \epsilon_{0} \alpha_{f}(x)^{t}\frac{C_{3}h_{H}(f^{n-t}(x))}{h_{H}(f^{n}(x)) }
+\frac{C(t)}{h_{H}(f^{n}(x))}\\[3mm]
& \leq \epsilon_{0} \alpha_{f}(x)^{t} \frac{C_{3} C_{2}(n-t)^{l} \alpha_{f}(x)^{n-t}}{C_{1} n^{l} \alpha_{f}(x)^{n}} +\frac{C(t)}{h_{H}(f^{n}(x))}\\[3mm]
& \leq \frac{C_{2}C_{3}}{C_{1}} \epsilon_{0}  +\frac{C(t)}{h_{H}(f^{n}(x))}.
\end{align*}
Note that
\begin{itemize}
\item $C_{1}, C_{2}, C_{3}$ are independent of $ \epsilon_{0}, n$;
\item $C(t)$ is independent of $n$ and $h_{H}(f^{n}(x))$ goes to infinity.
\end{itemize}
 Thus we are done.

\end{proof}

\begin{proof}[Proof of \cref{thm:divlchtnongen}]
We may take $h_{H}\geq1$ without loss of generality.
We use the notation in \cref{lem:keycalc}.

By the assumption, $e < \alpha_{f}(x)$.
Let $ \epsilon_{0}>0$ be any positive number.
Let $ \epsilon>0$ be such that $e+ \epsilon < \alpha_{f}(x)$.
For this $ \epsilon$, apply \cref{lem:keycalc} and take $t_{0}$.
Take a $t\geq t_{0}$ so that 
\[
\frac{(e+ \epsilon)^{t}}{ \alpha_{f}(x)^{t}} < \epsilon_{0}.
\]
Then again by \cref{lem:keycalc} (applied to $ \epsilon'=1$), there is a proper closed subset $Z \subset X$ such that
\begin{align*}
\sum_{v\in S}\lambda_{Y,v}(f^{n}(x)) \leq  \epsilon_{0} \alpha_{f}(x)^{t}( h_{H}(f^{n-t}(x)) - h_{K_{X}}(f^{n-t}(x))) + C(t)
\end{align*}
for all $n\geq t$ such that $f^{n-t}(x) \notin Z$.
We may take $h_{Y}$ so that if $f^{n}(x) \notin Y$, then $h_{Y}(f^{n}(x)) = \sum_{v\in M_{K}} \lambda_{Y, v}(f^{n}(x))$
(since the assumption does not depend on the choice of $h_{Y}$).
Thus 
\begin{align*}
&\sum_{v\notin S} \lambda_{Y, v}(f^{n}(x)) \\
&\geq h_{Y}(f^{n}(x)) - \epsilon_{0} \alpha_{f}(x)^{t} \left( h_{H}(f^{n-t}(x)) - h_{K_{X}}(f^{n-t}(x)) \right) - C(t)
\end{align*}
if $f^{n}(x) \notin Z' := f^{t}(Z) \cup Y$.

Now take constants $C_{1}, C_{2}>0$ and $l\geq0$ such that
\[
C_{1}n^{l} \alpha_{f}(x)^{n}\leq h_{H}(f^{n}(x)) \leq C_{2}n^{l} \alpha_{f}(x)^{n}
\]
for all $n\geq1$ (\cref{prop:growthht}).
Also, take a constant $C_{3}>0$ such that $h_{H}-h_{K_{X}} \leq C_{3}h_{H}$.
Then 
\begin{align*}
\frac{\sum_{v\notin S} \lambda_{Y, v}(f^{n}(x))}{ h_{H}(f^{n}(x)) } \geq 
\frac{h_{Y}(f^{n}(x))}{h_{H}(f^{n}(x))} - \epsilon_{0} \frac{C_{2}C_{3}}{C_{1}} - \frac{C(t)}{h_{H}(f^{n}(x))}
\end{align*}
if $f^{n}(x) \notin Z'$.
Since 
\begin{itemize}
\item $C_{1}, C_{2}, C_{3}$ are independent of $ \epsilon_{0}, n$;
\item $C(t)$ is independent of $n$ and $h_{H}(f^{n}(x))$ goes to infinity,
\end{itemize}
we are done.

\end{proof}

\begin{ex}
\cref{thm:divlcht} is not true without the assumption (3).
Let 
\[
f \colon \P^{2}_{\Q} \longrightarrow \P^{2}_{\Q} ; (X:Y:Z) \mapsto (X^{2}: p(Y,Z) : q(Y,Z))
\]
where $p,q  \in \Q[Y,Z]$ are coprime homogeneous  polynomials of degree $2$.
Let $D = (X=0) \subset \P^{2}$.
Then $f^{-1}(D)=D$ as sets and $f$ induces a surjective morphism $g \colon \P^{1}=D \longrightarrow D =\P^{1}$
defined by $p,q$.
It is easy to see that for a closed point $a \in D$,
\begin{align*}
& e_{f,+}(a) = 2    \quad \text{if $O_{f}(a)$ does not contain any periodic critical points of $g$;}\\
& e_{f,+}(a) >2  \quad  \text{if $O_{f}(a)$ contains a periodic critical point of $g$.}
\end{align*}

Let $p=Y^{2}+Z^{2}$, $q=3YZ$.
Then the critical points of $g$ are not periodic and therefore we have
\[
2\leq \max\{ e_{f,-}(y) \mid y \in D\} \leq \max\{ e_{f,+}(y) \mid \text{$y \in D$ closed point}\} = 2.
\]
Here we use \cref{thm:gignac} (\ref{e-<=e+}).
The point $(1:1:0)$ is a fixed point of $f$ and the tangent map at this point 
has eigenvalues $2, 3$, which are multiplicatively independent.
By the proof of \cite[Corollary 2.7]{abr}, there is a point $x = (1:a:b) \in \P^{2}(\Q)$ where $a, b$ are integers such that
the orbit $O_{f}(x)$ is Zariski dense.
Moreover, since $2^{i}3^{j}$ is not equal to a root of unity unless $i=j=0$,
the proof of \cite[Corollary 2.7]{abr} and \cite[Lemma 2.6 (2)]{abr} imply the orbit $O_{f}(x)$ is actually generic.
Note that we have $ \alpha_{f}(x) = 2$.

Take a local height associated with $D$ as
\[
\lambda_{D, \infty} = \log \frac{\max\{|X|,|Y|,|Z|\}}{|X|}.
\]
Then 
\begin{align*}
\frac{ \lambda_{D,v}(f^{n}(x))}{h_{\P^{2}}(f^{n}(x))}=1
\end{align*}
for all $n$.

%%%%%%%%%%%%%
\if0
Let $p=Y^{2}+Z^{2}$, $q=YZ$.
Then the critical points of $g$ are not periodic and therefore we have
\[
2\leq \max\{ e_{f,-}(y) \mid y \in D\} \leq \max\{ e_{f,+}(y) \mid \text{$y \in D$ closed point}\} = 2.
\]

Take a local height associated with $D$ as
\[
\lambda_{D, \infty} = \log \frac{\max\{|X|,|Y|,|Z|\}}{|X|}.
\]
Let $x = (1:1:2) \in \P^{2}(\Q)$.
Then $ \alpha_{f}(x) = 2$ and 
\begin{align*}
\frac{ \lambda_{D,v}(f^{n}(x))}{h_{\P^{2}}(f^{n}(x))}=1
\end{align*}
for all $n$.
Also we can check that the orbit $O_{f}(x)$ is Zariski dense in $\P^{2}$.
Moreover, $f$ induces a surjective self-morphism on $\A^{2} = \P^{2}\setminus D$ and
$O_{f}(x) \subset \A^{2}$.
Since Dynamical Mordell-Lang conjecture is proved for self-morphisms on $\A^{2}$ over $\QQ$ \cite{xie17},
the orbit $O_{f}(x)$ is generic.

\fi
%%%%%%%%%%%%%%

\end{ex}

\end{document}